\documentclass[11pt]{article}
\usepackage{amsmath}
\usepackage{amsthm}
\usepackage{amsfonts}
\usepackage{amssymb}
\usepackage{enumerate}
\usepackage{verbatim}
\usepackage{graphicx}

\theoremstyle{plain}
\newtheorem{theorem}{Theorem}
\newtheorem{lemma}[theorem]{Lemma}
\newtheorem{corollary}[theorem]{Corollary}
\newtheorem{proposition}[theorem]{Proposition}

\newtheorem*{lemma*}{Lemma~\ref{technical lemma}}
\newtheorem*{claim*}{Claim}

\theoremstyle{definition}
\newtheorem{definition}{Definition}
\newtheorem{conjecture}{Conjecture}

\newtheorem{problem}[conjecture]{Problem}

\theoremstyle{remark}
\newtheorem{remark}{Remark}

\title{Random subcube intersection graphs I: cliques and covering}
\author{Victor Falgas-Ravry  \thanks{Department of Mathematics, Vanderbilt University, 37240 Nashville, Tennessee, USA. Research supported by a grant from the Kempe foundation. Email: {\tt victor.falgas-ravry@vanderbilt.edu}} \and Klas Markstr\"om \thanks{Institutionen f\"or matematik och matematisk statistik, Ume{\aa}  Universitet, 901 87 Ume{\aa}, Sweden. Research supported by a grant from the Swedish Research Council. Email: {\tt klas.markstrom@math.umu.se}}}

\begin{document}
\maketitle
\begin{abstract}
We study \emph{random subcube intersection graphs}, that is, graphs obtained by selecting a random collection of subcubes of a fixed hypercube $Q_d$ to serve as the vertices of the graph, and setting an edge between a pair of subcubes if their intersection is non-empty. Our motivation for considering such graphs is to model `random compatibility' between vertices in a large network.

For both of the models considered in this paper, we determine the thresholds for covering the underlying hypercube $Q_d$ and for the appearance of $s$-cliques. In addition we pose a number of open problems.
\end{abstract}

\section{Introduction}
In this paper we introduce and study two models of \emph{random subcube intersection graphs}. These are random graph models obtained by (i) selecting a random collection of subcubes of a fixed hypercube $Q_d$, to serve as the vertices of the graph, and (ii) setting an edge between a pair of subcubes if their intersection is non-empty. Our basic motivation for studying these random graphs is that they give a model for `random compatibility' between vertices.   Before we make our models mathematically precise let us consider some examples of the applications we have in mind.

A first example is the random $k$-SAT problem, which has attracted the attention of both physicists \cite{Kirkpatrick19941297} and mathematicians \cite{Achlioptas2005759} for many decades.  In this problem we have a set of $n$ Boolean variables  and some number $m$ of Boolean clauses on $k$ variables are 
chosen at random. Each clause forbids exactly one of the $2^k$ possible assignments to the $k$ variables in the clause and the question of interest  is whether there exists an assignment to the $n$ variables which is compatible with all the clauses. It is known that there is a sharp threshold~\cite{MR1678031} for satisfiability with respect to the number $m$ of clauses chosen and that for large $k$~\cite{C-O} this threshold is located at approximately $m=n2^k\ln(k)$.  It is conjectured that for all fixed $k$ there exists a constant $c_k$ such that the satisfiability threshold is  asymptotically $ c_k n$ (a proof of this conjecture for all sufficiently large values of $k$ has been announced recently~\cite{DSS}).

The random $k$-SAT problem is a problem about random subcubes.  Given a clause $C$ the assignments which are incompatible with $C$ are given by the subcube where the $k$ variables in $C$ are assigned the values forbidden by $C$. This is an $(n-k)$-dimensional subcube of the $n$-dimensional cube of all possible assignments, and a collection of clauses is unsatisfiable if and only if the union of the corresponding subcubes contains all vertices of the $n$-cube.    The random $k$-SAT problem is thus equivalent to finding the threshold for covering all the vertices of a cube by a collection of random subcubes.   This example also suggests that a mathematical analysis of the covering problem will be harder for subcubes than for e.g. the usual independent random intersection graph models (where it is analogous to the classical coupon collector problem).

A second example of applications we have in mind comes from social choice theory. Suppose we have a society $V$ which is faced with $d$ political issues, each which can be resolved in a binary fashion. We represent the two policies possible on each issue by $0,1$, and the family of all possible sets of policies by a $d$-dimensional hypercube $Q_d$.

Individual members of the society may have fixed views on some issues, but may be undecided or indifferent on others. We can thus associate to each citizen $v\in V$ a subcube of acceptable policies $f(v)$ in a natural way. The subcube intersection graph $G$ arising from $(V, Q_d, f)$ then represents political agreement within the society: $uv$ is an edge of $G$ if and only if the citizens $u$ and $v$ can agree on a mutually acceptable set of policies.

A key characteristic of subcube intersection graphs is that they possess the \emph{Helly property}: if we have $s$ subcubes $f(v_1), f(v_2), \ldots f(v_s)$ of $Q_d$ which are pairwise intersecting, then their total intersection $\bigcap_{i=1}^s f(v_i)$ is non-empty (this is an easy observation, already made in~\cite{JohnsonMarkstrom12}). A consequence of this fact is that in the model for political agreement described above, $s$-cliques represent $s$-sets of citizens able to agree on a mutually acceptable set of policies and, say, unite their forces to promote a common political platform. This example 
motivates our study of the clique number in (random) subcube intersection graphs.

There are many other examples of compatibility graphs naturally modeled by subcube intersection graphs. Some closely resemble the one above: the work of matrimonial agencies or the assignment of room-mates in the first year at university for instance naturally lead to the study of such compatibility graphs. 
Another class of examples can be found in the medical sciences.  For kidney or blood donations, several parameters must be taken into account to determine whether a potential donor--receiver pair is compatible. Large random subcube intersection graphs provide a way of modeling these compatibility relations over a large pool of donors and receiver, and of identifying efficient matching schemes.

\subsection{The models}
Let us now describe our models more precisely. We begin with some basic definitions and notation.

\begin{definition}[Intersection graphs]
A \emph{feature system} is a triple $(V, \Omega, f)$, where $V$ is a set of vertices, $\Omega$ is a set of \emph{features}, and $f$ is a function mapping vertices in $V$ to subsets of $\Omega$. Given a vertex $v\in V$, we call $f(v)\subseteq \Omega$ its \emph{feature set}. We construct a graph $G$ on the vertex-set $V$ from a feature system $(V, \Omega, f)$ by placing an edge between $u,v \in V$ if their feature sets $f(u), f(v)$ have non-empty intersection. We call $G$ the \emph{intersection graph} of the feature system $(V, \Omega,f)$.
\end{definition}
In this paper we shall study intersection graphs where $\Omega$ and the feature sets $\{f(v): \ v \in V\}$ have some additional structure. Namely, $\Omega$ shall be a high-dimensional hypercube $Q_d$ and the feature sets will consist of subcubes of $Q_d$. 
\begin{definition}[Hypercubes and subcubes]
The \emph{$d$-dimensional hypercube} is the set $Q_d=\{0,1\}^d$. A \emph{$k$-dimensional subcube} of $Q_d$ is a subset obtained by fixing $(d-k)$-coordinates and letting the remaining $k$ vary freely. We may regard subcubes of $Q_d$ as elements of $\{0,1, \star\}^d$, where $\star$ coordinates are free and the $0,1$ coordinates are fixed.
\end{definition}
We shall define two models of random subcube intersection graphs. Both of these are obtained by randomly assigning to each vertex $v \in V$ a feature subcube $f(v)$ of $Q_d$ and then building the resulting intersection graph.
\begin{definition}[Uniform model]
Let $V$ be a set of vertices. Fix $k,d\in \mathbb{N}$ with $k\leq d$. For each $v \in V$ independently select a $k$-dimensional subcube $f(v)$ of $Q_d$ uniformly at random, and set an edge between $u,v \in V$ if $f(u) \cap f(v) \neq \emptyset$. Denote the resulting random subcube intersection graph by $G_{V, d, k}$.
\end{definition}
\begin{definition}[Binomial model]
Let $V$ be a set of vertices. Fix $d\in\mathbb{N}$ and $p\in[0,1]$. For each $v\in V$ independently select a subcube $f(v)\in \{0,1, \star\}^d$ at random by setting $(f(v))_i=\star$ with probability $p$ and $(f(v))_i=0,1$ each with probability $\frac{1-p}{2}$ independently for each coordinate $i \in \{1,\ldots d\}$ (we refer to such a subcube as a \emph{binomial random subcube}). Denote the resulting random subcube intersection graph by $G_{V,d, p}$. 
\end{definition}

\begin{remark}\label{binomial d as intersection of binomial 1}
We may view $G_{V,d,p}$ as the intersection of $d$ independent copies of $G_{V,1,p}$ on a common vertex-set $V$. Indeed an edge $uv$ of $G_{V,d,p}$ is present if and only if for each of the $d$ dimensions of $Q_d$ we have that  $f(u)$ and $f(v)$, seen as vectors, are identical or at least one of them is $\star$. The graph $G_{V,1,p}$ is itself rather easy to visualise: we first randomly colour the vertices in $V$ with colours from $\{0,1, \star\}$, and then remove from the complete graph on $V$ all edges between vertices in colour $0$ and vertices in colour $1$.
\end{remark}

\subsection{Degree distribution, edge-density and relation to other models of random graphs}
Our two models of random subcube intersection graphs bear some resemblance to previous random graph models. To give the reader some early intuition into the nature of random subcube intersection graphs, we invite her to consider the degree distributions and edge-densities found in them, and to contrast them with models of random graphs with similar degree distributions and edge-densities.

Let us first note that in order to get an random model which is both structurally interesting and amenable to asymptotic analysis we typically consider the case where $d \rightarrow  \infty$, and the other parameters are functions of $d$.

The degree of a given vertex in the uniform model $G_{V,d, k}$ is a binomial random variable with parameters $\vert V \vert-1$ and $q$, where $q$ is the probability that two uniformly chosen $k$-dimensional subcubes of $Q_d$ meet. If $k=k(d)=\lfloor \alpha d\rfloor$ for some fixed $\alpha \in (0,1)$, then one can show $q=q(\alpha)= e^{-f(\alpha)d+o(d)}$, where
\begin{align*}
f(\alpha)= &2\log\left(\alpha^{\alpha}(1-\alpha)^{1-\alpha}\right) - \left(\sqrt{(1-\alpha)^2+\alpha^2}-1+\alpha\right)\log\left(\sqrt{(1-\alpha)^2+\alpha^2}-1+\alpha\right)\\ &-2\left(1-\sqrt{(1-\alpha)^2+\alpha^2} \right)\log \left(1-\sqrt{(1-\alpha)^2+\alpha^2} \right)\\
&-\left(\sqrt{(1-\alpha)^2+\alpha^2}-\alpha\right)\log\left(2\sqrt{(1-\alpha)^2+\alpha^2}-2\alpha\right).
\end{align*}
This expression is not, however, terribly instructive.

The quantity $q$ is also the edge-density of $G_{V,d,k}$. When $\vert V\vert=n$, the appropriate random graph to compare and contrast it with is thus an Erd\H{o}s--R\'enyi random graph $G(n,q)$ with edge probability $q$. However
$G_{V, d, k}$ displays some significant \emph{clustering}: our results can be used to show for instance that dependencies between the edges cause triangles to appear well before we see a linear number of edges, in contrast to the Erd\H{o}s--R\'enyi model $G(n,q)$.

The edge-density of the binomial model $G_{V,d,p}$ is easy to compute: it is exactly $\left(1-\frac{(1-p)^2}{2}\right)^d=e^{-d\log \left(\frac{2}{1+2p-p^2}\right)}$. The degree distribution of $G_{V, d, p}$ is more complicated, however. Increasing the dimension of a subcube by $1$ doubles its volume inside $Q_d$, so that larger subcubes expect much larger degrees. The number of feature subcubes from our graph met by a fixed subcube of dimension $\alpha d$ is a binomial random variable with parameters $\vert V\vert-1$ and $\left(\frac{1+p}{2}\right)^{(1-\alpha)d}$. The number of vertices in $V$ whose feature subcubes have dimension $\alpha d$ is itself a binomial random variable with parameters $\vert V \vert$ and $\binom{d}{\alpha d}p^{\alpha d}(1-p)^{(1-\alpha)d}$. As in this paper we will typically be interested in the case where $d$ is large and $V$ has size exponential in $d$, we will expect to see some feature subcubes with dimension much larger or much smaller than $pd$. This will have a noticeable effect on the properties of the graph $G_{V,d, p}$.

\begin{figure}
\includegraphics[scale=0.5]{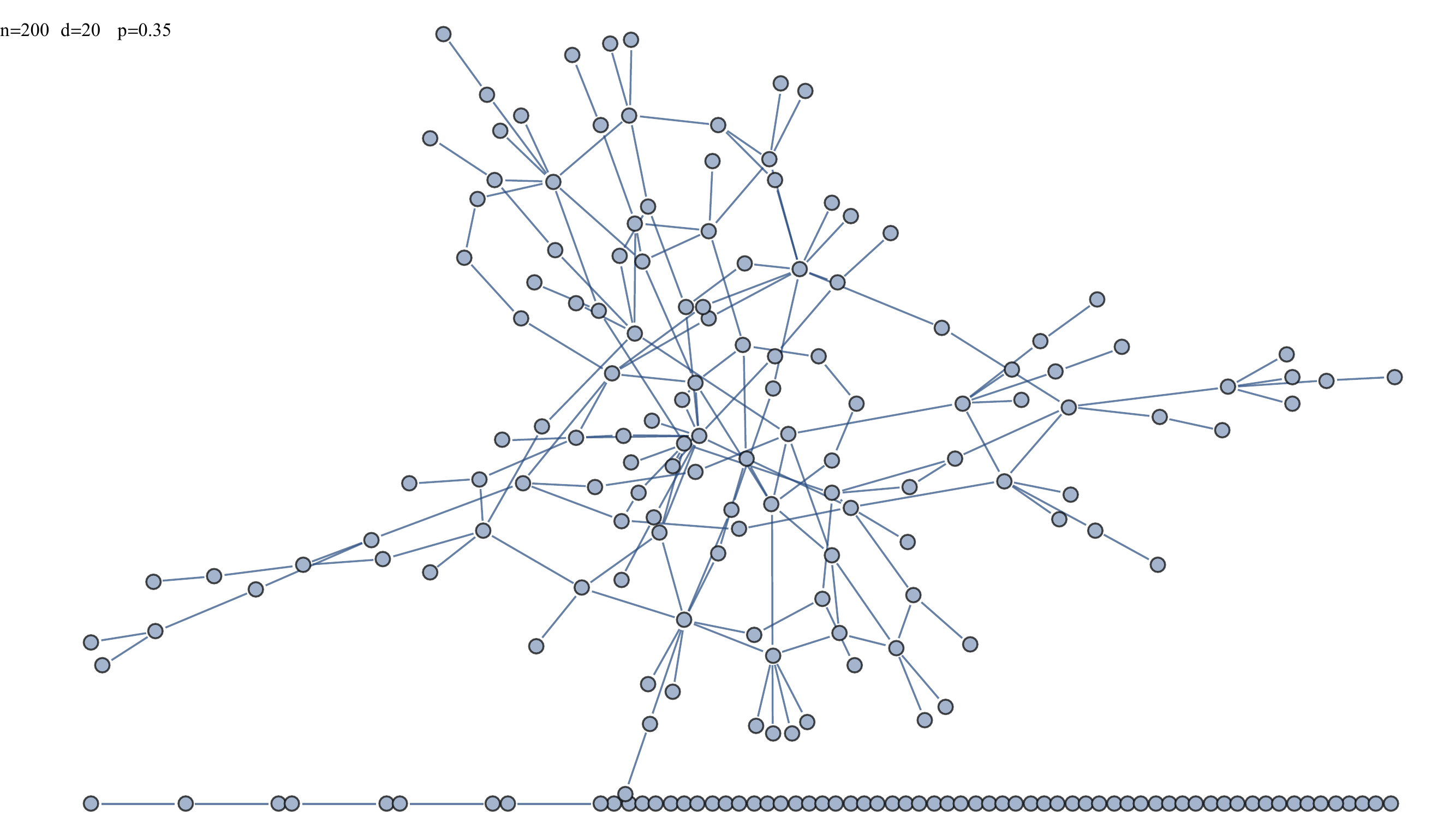}
\caption{An example of the binomial random subcube intersection graph model $G_{V,d,p}$ with $\vert V\vert=200$, $d=20$ and $p=0.35$. The row of vertices at the bottom right are all isolated.}
\end{figure}

Among the random graph models studied in the literature, $G_{V,d, p}$ in many ways resemble the multi-type inhomogeneous random graphs studied in~\cite{BollobasJansonRiordan07},  which also have vertices of several different types and differing edge probabilities, though we should point out there are significant differences. First of all some `types' corresponding to vertices with unusually large or unusually small feature subcubes will have only a sublinear (and random) number of representatives. Secondly, the binomial model shares the clustering behaviour of the uniform model (see Remark~\ref{clustering binomial case}), differentiating it from the models considered in~\cite{BollobasJansonRiordan07}. We note that a further general model for inhomogeneous random graphs with clustering was introduced by Bollob\'as, Janson and Riordan in~\cite{BollobasJansonRiordan11}, for which this second point does not apply.

Finally, let us mention the standard models of \emph{random intersection graphs}. Write $[m]$ for the discrete interval $\{1,2 , \ldots m\}$. In the \emph{binomial random intersection graph model} $\mathcal{G}(V,[m],p)$, each vertex $v \in V$ is independently assigned a random feature set $f(v) \subseteq [m]$. This feature set is obtained by including $j \in [m]$ into $f(v)$ with probability $p$ and leaving it out otherwise independently at random for each feature $j \in [m]$. Edges are then added between all pairs of vertices $u,v \in V$ with $f(u) \cap f(v) \neq \emptyset$ to obtain a random intersection graph on $V$. A variant on this model is to choose feature sets $f(v)$ uniformly at random from the $k$-subsets of $[m]$; this yields the \emph{uniform random intersection graph model} $\mathcal{G}(V,[m], k)$.

While these two random intersection graph models bear some resemblance (in terms of clustering, for example) to our random subcube intersection graph models, there are also some significant differences due to the underlying structure of our feature sets. Let us note amongst other things that random intersection graphs do not have the Helly property, and that the effects on the degree of increasing the size of a feature set by $1$ in a binomial random intersection graph are far less dramatic than the effects of increasing the dimension of a feature subcube by $1$ in a binomial random subcube intersection graph. In particular, the binomial random subcube intersection graph model $G_{V,d,p}$ 
 has a much more dramatic variation of degrees than its non-structured counterpart $\mathcal{G}(V,[m],p)$.

We end this section by noting that there has been some interest in another model of `structured' random intersection graphs, namely \emph{random interval graphs}. The idea here is to associate to each vertex $v \in V$ a feature interval $f(v)=I_v=[a_v,b_v]\subseteq [0,1]$ at random and to set an edge between $u,v \in V$ whenever $I_u \cap I_{v}\neq \emptyset$. Here `at random' means the intervals are generated by independent pairs of uniform $U(0,1)$ random variables, which serve as the endpoints. A $d$-dimensional version of this model also exists, where we associate to each vertex a $d$-dimensional box lying inside $[0,1]^d$. This gives rise to \emph{(random) $d$-box graphs}.

In the setting of interval or $d$-box graphs, we do have the Helly property. The random interval and random $d$-box graph models are however quite different from the random subcube intersection graphs we study in this paper. Indeed, for $d=1$ subcube intersection graphs can be viewed as interval graphs where the feature intervals $I_v$ are restricted to a small set of possible values, for example $[0,1]$, $[0,1/3]$ and $[2/3,1]$ (to correspond to $\star$, $0$ and $1$ respectively). This naturally leads to a very different structure.

\subsection{Previous work on random intersection graphs and subcube intersection graphs}
Subcube intersection graphs were introduced by Johnson and Markstr\"om~\cite{JohnsonMarkstrom12}, with motivation coming from the example in social choice theory we discussed above. They studied cliques in subcube intersection graphs from an extremal perspective, obtaining a number of results on Ramsey- and Tur\'an-type problems and providing a counterpoint to the probabilistic perspective of the work undertaken in this paper.

The random intersection graph models $\mathcal{G}(V,[m],p)$ and $\mathcal{G}(V,[m],k)$ we presented in the previous subsection have for their part received extensive attention from the research community since they were introduced by Karo{\'n}ski, Scheinerman and Singer-Cohen~\cite{KaronskiScheinermanSingerCohen99} and Singer-Cohen~\cite{SingerCohen95}. By now, many results are known on their connectivity~\cite{BlackburnGerke09, GodehardtJaworskiRybarczyk07,Rybarczyk11b,SingerCohen95}, hamiltonicity~\cite{BloznelisRadavicius10,EfthymiouSpirakis05}, component evolution~\cite{Behrisch07,BloznelisJaworskiRybarczyk09,Rybarczyk11b}, clique number~\cite{BloznelisKurauskas13,KaronskiScheinermanSingerCohen99,RybarczykStark10}, independence number~\cite{NikoletseasRaptopoulosSpirakis08}, chromatic number~\cite{BehrischTarazUeckerdt09,NikoletseasRaptopoulosSpirakis09}, degree distribution
\cite{Stark04} and near-equivalence to the Erd\H os--R\'enyi model $\mathcal{G}_{n,p}$ for some range of the parameters~\cite{FillScheinermanSingerCohen00,Rybarczyk11}, amongst other properties. Even more recently, there has been interest in obtaining versions of the results cited above for inhomogeneous random intersection graph models.

Finally there has been some work on random intersection graphs and $d$-box graphs that runs somewhat parallel to the work of Johnson and Markstr\"om and of this paper. From an extremal perspective, sufficient conditions for the existence of large cliques in $d$-box graphs were investigated by Berg, Norine, Su, Thomas and Wollan~\cite{BergNorineSuThomasWollan10} in the context of models for social agreement and approval voting, while random interval graphs were introduced by Scheinerman~\cite{Scheinerman88}, and have been extensively studied~\cite{DiaconisHolmesJanson11, GodehardtJaworski98, Pippenger98,Scheinerman90}.

\subsection{Results of this paper}
In this paper we study the behaviour of the binomial and uniform subcube intersection models when $d$ is large (see Remark~\ref{choice of parameters} below for a discussion of the constant $d$ case). We study two main properties, that of containing a clique of size $s=s(d)$, and that of covering the entirety of the underlying hypercube $Q_d$ with the union $\bigcup_{v\in V}f(v)$ of the feature subcubes.

Both of these properties are closed under the addition of vertices to $V$ (or, equivalently, of subcubes $f(v)$ to the family of feature subcubes). The question is then how large $V$ needs to be for these properties to hold \emph{with high probability} (whp), that it to say with probability tending to $1$ as $d\rightarrow \infty$. In the case of covering, this question can be thought of as a structured variant of the classical coupon collector problem (see the discussion at the beginning of Section~\ref{binomial: at covering}).

Formally, we take a dynamic view of our models: for fixed $p, \alpha \in [0,1]$ we consider a nested sequence of vertex sets $V_1\subset V_2 \subset \ldots$, with $\vert V_n\vert=n$, and corresponding nested sequences of binomial random subcube intersection graphs $B_n=G_{V_n, d, p}$ and uniform random subcube intersection graphs $U_n=G_{V_n, d, \lfloor \alpha d\rfloor}$. 
\begin{definition}
Let $\mathcal{P}$ be a property of subcube intersection graphs that is closed with respect to the addition of vertices. The \emph{hitting time} $N^b_{\mathcal{P}} = N^b_{\mathcal{P}}(d,p)$ for $\mathcal{P}$ for the binomial sequence $(B_n)_{n\in\mathbb{N}}$ is 
\[N^b_{\mathcal{P}}:=\min\left\{n \in \mathbb{N}: \   B_n \in \mathcal{P} \right\}.\]
Similarly, we define the \emph{hitting time} $N^u_{\mathcal{P}} = N^u_{\mathcal{P}}(d,\alpha)$ for $\mathcal{P}$ for the uniform sequence $(U_n)_{n\in \mathbb{N}}$ is 
\[N^u_{\mathcal{P}}:=\min\left\{n \in \mathbb{N}: \   U_n \in \mathcal{P} \right\}.\]
\end{definition}
In this paper, we restrict our attention to the binomial model $G_{V,d,p}$ with $p\in (0,1)$ fixed, and to the uniform model $G_{V, d, k}$ with $k=k(d)=\lfloor \alpha d \rfloor$ for $\alpha \in (0,1)$ fixed. In both cases, the interesting behaviour occurs when $\vert V\vert = e^{xd}$ for $x$ bounded away from $0$. We thus typically use this number $x$ as a parameter, rather than the actual number $n=\vert V\vert$ of vertices in the graph. Our aim is to establish concentration of the exponent of the hitting time. We thus make the following definitions:
\begin{definition}
Let $\mathcal{P}$ be a property of subcube intersection graphs that is closed with respect to the addition of vertices. A real number $t\geq 0$ is a \emph{threshold} for $\mathcal{P}$ in the binomial model (with parameter $p$) if
\[\frac{\log \left(N^b_{\mathcal{P}}(d,p)\right)}{d} \rightarrow t\]
in probability as $d\rightarrow \infty$. Similarly, we say that $t\geq 0$ is a \emph{threshold} for $\mathcal{P}$ in the uniform model (with parameter $k=\lfloor\alpha d\rfloor$) if 
\[\frac{\log \left(N^u_{\mathcal{P}}(d,\alpha)\right)}{d} \rightarrow t\]
in probability as $d\rightarrow \infty$.
\end{definition}
In other words, $t_{\mathcal{P}}>0$ is a threshold for the binomial model if for any sequence of vertex sets $V=V(d)$
\[\lim_{d\rightarrow \infty} \mathbb{P}(G_{V, d, p}\in \mathcal{P})=\left\{\begin{array}{ll}
0 & \textrm{if $\vert V(d)\vert \leq e^{xd}$ for some $x<t$,}\\
1& \textrm{if $\vert V(d)\vert \geq e^{xd}$ for some $x>t$,}\end{array}  \right.\]
and a similar statement holds in the case of the uniform model.

Our main results are determining the thresholds for the appearance of cliques and for covering the ambient hypercube in both the binomial and the uniform model. 
In most cases we also give some slightly more precise information about the thresholds, going into the lower order terms. We show in particular that around the covering threshold, the clique number of our models undergoes a transition: below the covering threshold, the clique number is whp of order $O(1)$; close to the covering threshold, it is whp of order a power of $d$; finally above the covering threshold, it is whp of order exponential in $d$.

Our paper is structured as follows. In Section~\ref{binomialsection}, we state and prove our results for the binomial model. In Section~\ref{uniformsection} we use these to obtain our results for the uniform model. Finally in Section~\ref{remarkssection} we discuss small $p$ and large $p$ behaviour, and end with a number of open problems.
\begin{remark}\label{choice of parameters}
In this paper, as we have said, we are focussing on our models in the case where $d\rightarrow\infty$. What happens when $d$ is fixed and the number of vertices goes to infinity? In some applications, this may be a more relevant choice of parameters. The asymptotic behaviour in this case is however much simpler. Indeed, let $d$ be fixed and let $U$ be the family of all subcubes of $Q_d$. We may define a subcube intersection graph $G^d$ on $U$ by setting an edge between two subcubes if their intersection is non-empty. The binomial model $G_{\vert V\vert ,d,p}$ is then just a \emph{random weighted blow-up} of $G^d$: each vertex $v$ of $G^d$ is replaced by a clique with a random size $c_v$, where $\sum_v c_v =\vert V\vert $, and by standard Chernoff bounds $c_v=(1+o(1))p_v\vert V\vert$ for every $v$, where $p_v$ is the probability that a binomial random cube is equal to $v$.  
Thus knowledge of the  finite graph $G^d$ will give us essentially all the information we could require concerning the graph $G_{V,d,p}$ as $\vert V \vert \rightarrow \infty$. 
  
Similarly, the asymptotic behaviour of $G_{V,d,k}$ for $d$ fixed can be inferred from the properties of the intersection graph $G_k^d$ of the $k$-dimensional subcubes of $Q_d$. We note that this latter graph $G_k^d$ may be thought of as a subcube analogue of (the complement of) a Kneser graph, and may be interesting in its own right as a graph theoretical object; this is however outside the scope of the present paper.
\end{remark}

\subsection{A note on approximations and notation}
Throughout this paper we shall need some standard approximations. In particular we shall often use 
$\binom{m}{\beta m}=e^{-m\log \left(\beta^{\beta}(1-\beta)^{1-\beta}\right)+O(\log m) }$ (for $\beta \in (0,1)$ fixed)
and $(1-\eta)^m=e^{-\eta m + O(\eta^2 m)}$ (for $\eta=o(1)$). We will also use  the notation $f(n) \ll g(n)$ to denote that $f(n) = o(g(n))$, and $f(n) \gg g(n)$ to denote that $g(n) = o(f(n))$.

\section{The binomial model}\label{binomialsection}

\subsection{Summary}
In this section, we prove our results for the binomial model. Denote by $K_s$ the complete graph on $s$ vertices. Recall that the \emph{clique number} $\omega(G)$ of a graph $G$ is the largest $s$ such that $G$ contains a copy of $K_s$ as a subgraph.

\begin{theorem}\label{binomial s<<d/log d}
Let $p \in (0,1)$ and $\varepsilon>0$  fixed. Let $s=s(d)$ be a sequence of non-negative integers with $s(d)=o\left(\frac{d}{\log d}\right)$. Set 
\[t_{K_s}(p)= -\frac{1}{s} \log \left(2\left(\frac{1+p}{2}\right)^s- p^s \right).\]
Then for every sequence of vertex sets $V(d)$ with $x(d)=\frac{1}{d} \log \vert V(d) \vert$,
\[\lim_{d\rightarrow \infty}\mathbb{P}\left(G_{V,d, p} \textrm{ contains a }K_s\right)=\left\{\begin{array}{ll} 0 & \textrm{if } x(d) \leq t_{K_s}(p)+ \frac{\log s}{d} - \varepsilon\frac{ \log d}{d}\\
1 & \textrm{if } x(d) \geq t_{K_s}(p)+ \frac{2\log s}{d} + \varepsilon\frac{ \log d}{d}.\end{array}  \right. \]
\end{theorem}
\begin{corollary}\label{binomial s constant clique threshold}
Let $p\in (0,1)$ and $s\in \mathbb{N}$ be fixed. The threshold for the appearance of $s$-cliques in $G_{V,d,p}$ is
\[t_{K_s}(p)= \log \frac{2}{1+p} - \frac{1}{s} \log \left(2 - \left(\frac{2p}{1+p}\right)^s \right).\]
\end{corollary}
\begin{remark}\label{clustering binomial case}
As we shall see in the proof of Theorem~\ref{binomial s<<d/log d}, from the moment it becomes non-zero, the number of edges in $G_{V,d,p}$ remains concentrated about its expectation $e^{2(x-t_{K_2})d +o(d)}$. 
If there was no clustering in $G_{V,d,p}$, that is, if cliques appeared no earlier than they would in the Erd\H{o}s--R\'enyi model with parameter $e^{-2t_{K_2}d}$, then we would expect $s$-cliques to appear roughly when $x=(s-1)t_{K_2}$.

However, it is the case that $t_{K_s}<(s-1)t_{K_2}$ for all $p\in [0,1)$ and all $s\geq 3$. This is an exercise in elementary calculus.
In particular, $s$-cliques appear much earlier than we would expect them to given the edge-density of our binomial random subcube intersection graphs. Indeed, letting $p\rightarrow 0$, we have by Corollary~\ref{binomial s constant clique threshold} that for $s\geq 3$
\[t_{K_s}(p)= \left(1-\frac{1}{s}\right)\log 2  -p +O(p^2),\]
while the threshold for having a linear number of edges is $2t_{K_2}=\log 2 -2p +O(p^2)$, which is strictly larger provided $p$ is chosen sufficiently small. Thus for every $s\in \mathbb{N}$, there exists $p_s\in [0,1]$ such that for all fixed $p\in [0, p_s]$, whp we see $s$-cliques appear in $G_{V,d,p}$ before we have a linear number of edges. This stands in stark contrast to the situation for the Erd\H{o}s--R\'enyi model.
\end{remark}

\begin{theorem}\label{binomial covering threshold}
Let $p \in (0,1)$ and $\varepsilon >0$ be fixed. Let $V=V(d)$ be a sequence of vertex sets with $x(d)= \frac{1}{d} \log \vert V(d) \vert$. Then, for the binomial model $G_{V,d,p}$,
\[\lim_{d\rightarrow \infty}\mathbb{P}\left(\bigcup_{v \in V} f(v)=Q_d\right)=\left\{\begin{array}{ll} 0 & \textrm{if } x(d) \leq \log \frac{2}{1+p}+\frac{\log d}{d}+\frac{\log \left(\log 2-\varepsilon\right)}{d}\\
1 & \textrm{if } x(d) \geq \log \frac{2}{1+p}+\frac{ \log d}{d}+\frac{\log \left(\log 2+\varepsilon\right)}{d}.\end{array}  \right.\]
\end{theorem}
\begin{corollary}
Let $p \in (0,1)$ be fixed. Then the threshold for covering the ambient hypercube $Q_d$ with the feature subcubes from $G_{V,d,p}$ is
\[t_{\textrm{cover}}(p)= \log \frac{2}{1+p}.\]
\end{corollary}
\begin{remark}\label{covering as limit of clique thresholds} $\lim_{s \rightarrow \infty} t_{K_s}(p)=t_{\textrm{cover}}(p)$.
\end{remark}

\begin{theorem}\label{binomial s>>d/logd}
Let $p\in (0,1)$ and $\varepsilon>0$ be fixed, and let $s=s(d)$ be a sequence of integers with $s\gg d/\log d$.
Then for every sequence of vertex sets $V(d)$ with $x(d)=\frac{1}{d}\log \vert V(d) \vert$,
\[\lim_{d\rightarrow \infty}\mathbb{P}\left(G_{V,d, p} \textrm{ contains a }K_s\right)=\left\{\begin{array}{ll} 0 & \textrm{if } x(d) \leq \log\frac{2}{1+p}+ \frac{\log s}{d} - \frac{\varepsilon \log d}{d}\\
1 & \textrm{if } x(d) \geq \log\frac{2}{1+p}+ \frac{\log s}{d} + \frac{\varepsilon}{d}.  \end{array}\right. \]
Further, if $\frac{s(d)}{d}\rightarrow \infty$ as $d\rightarrow \infty$, then we may improve the lower bound on the appearance of $s$-cliques to $x(d)\leq \log\frac{2}{1+p}+ \frac{\log s}{d} - \frac{\varepsilon}{d}$.
\end{theorem}

Theorem~\ref{binomial s<<d/log d} is proved in Section~\ref{binomial: below covering}, where in addition we prove some key results on the dimension of the feature subcubes of the vertices in the first $s$-clique to appear in our graph. These will be needed in Section~\ref{uniformsection} when we study the uniform model. Theorem~\ref{binomial covering threshold} is proved in Section~\ref{binomial: at covering}, while Theorem~\ref{binomial s>>d/logd} is proved in Section~\ref{binomial: above covering}. Our results give whp lower and upper bounds on certain hitting times, and their proofs are split accordingly into two parts, one for each direction.

Before we proceed to the proofs, let us remark that our results imply that the clique number $\omega(G_{V,d,p})$ undergoes a transition around the covering threshold. 
\begin{corollary}
Let $p\in (0,1)$. Let $V=V(d)$ be a sequence of vertex-sets and $x(d)= \frac{1}{d} \log \vert V(d)\vert$. The following hold:
\begin{itemize} 
\item if there is $s \in \mathbb{N}$ and $\varepsilon>0$  such that $t_{K_s}+\varepsilon<x <t_{K_{s+1}}$,
then whp $\omega(G_{V,d,p})=s$;
\item if there is $s \in \mathbb{N}$ such that  
 $x=t_{K_s}+o(1)$,
 then whp $\omega(G_{V,d,p})\in \{s, s-1\}$;
\item if there is $\gamma>0$ such that  
 $x=x(d)=t_{\textrm{cover}}+ \gamma \frac{\log d}{d}+o\left(\frac{\log d}{d}\right)$,
 then whp $\omega(G_{V,d,p})$ has order $d^{\gamma+o(1)}$;
\item if there is $c>0$ such that  
 $x=x(d)=t_{\textrm{cover}}+c+o(1)$, 
 then whp $\omega(G_{V,d,p})$ has order $e^{cd+o(d)}$.
\end{itemize}
\end{corollary}

\subsection{Below the covering threshold}\label{binomial: below covering}
\begin{proof}[Proof of Theorem~\ref{binomial s<<d/log d}]
Without loss of generality we may assume $V=[n]$. Set $x=\frac{1}{d}\log n$, and let $\varepsilon>0$ be fixed. Let $s=s(d)$ be a sequence of non-negative integers with $s(d)=o\left(d/\log d\right)$.

Let $q(s,d)$ denote the probability that a given $s$-set of vertices induces an $s$-clique in $G_{[n],d, p}$. By Remark~\ref{binomial d as intersection of binomial 1}, we have that
\begin{align*}
q(s,d)&= q(s,1)^d= \left(2\left(\frac{1+p}{2}\right)^s- p^s\right)^d= \exp \left(-s d t_{K_s}(p)\right).
\end{align*}
Let $X=X(d)$ be the random variable denoting the number of copies of $K_s$ in $G_{[n], d, p}$.

\noindent \textbf{Lower bound:} suppose $x\leq t_{K_s}(p)+ \frac{\log s}{d} - \varepsilon\frac{\log d}{d}$. Then
\begin{align*}
\mathbb{E}X= \binom{n}{s}q(s,d)&=\exp\left(sd x - s\log s - sd t_{K_s}(p)+O(s)\right)\\
&\leq\exp\left(-\varepsilon s \log d +O(s)\right)=o(1).
\end{align*}
It follows by Markov's inequality that whp $X=0$ and $G_{[n],d,p}$ contains no $s$-clique, proving the first part of the theorem.

 \noindent \textbf{Upper bound:} suppose $x \geq t_{K_s}(p)+ \frac{2\log s}{d} + \varepsilon\frac{\log d}{d}$. We have
\begin{align*}
\mathbb{E}X= \binom{n}{s}q(s,d)&\geq \left(\frac{n}{s}\right)^sq(s,d)\geq\exp\left(\varepsilon s \log d \right)\gg 1.
\end{align*}
We use Chebyshev's inequality to show that $X$ is concentrated about this value (and hence that whp $G_{[n],d,p}$ contains an $s$-clique).

Fix $i:\ 0 \leq i \leq s$. Let $A,B$ be two $s$-sets of vertices meeting in exactly $i$ vertices. Using Remark~\ref{binomial d as intersection of binomial 1} and the inclusion-exclusion principle, we compute the probability $b_i$ that both $A$ and $B$ induce a copy of $K_s$ in $G_{[n],d,p}$:
\begin{align*}
b_i&=\left(2\left(\frac{1+p}{2}\right)^{2s-i}+ 2 \left(\frac{1+p}{2}\right)^{2s-2i}p^i - 4\left(\frac{1+p}{2}\right)^{s-i}p^s + p^{2s-i}\right)^d\\
&=\left(\frac{1+p}{2}\right)^{(2s-i)d}\left(2+2 \left(\frac{2p}{1+p}\right)^i -4 \left(\frac{2p}{1+p}\right)^s+ \left(\frac{2p}{1+p}\right)^{2s-i} \right)^d.
\end{align*}
(Note $b_0=q(s,d)^2$ and $b_s=q(s,d)$.) Now, 
\begin{align*}
\mathbb{E}X^2&= \binom{n}{s} \sum_{i=0}^s \binom{n-s}{s-i}\binom{s}{i} b_i.
\end{align*}
We claim that the dominating contribution to this sum comes from the $i=0$ term. Indeed, for $s(d)=o\left(\frac{d}{\log d}\right)$, $d$ large and $x\geq t_{K_s}+ \frac{2\log s}{d} + \varepsilon\frac{\log d}{d}$,
\begin{align*}
 \frac{\binom{n-s}{s-i}\binom{s}{i}}{\binom{n-s}{s}}\frac{b_i}{b_0}
&\leq  \frac{2s^{2i}}{n^{i}}\frac{b_i}{b_0} \qquad \qquad \textrm{(provided $d$ is sufficiently large)} \notag \\
&\ =2s^{2i} \exp \Bigl(d\Bigl[-ix -(2s-i)\log \frac{2}{1+p}\notag\\
& \qquad + \log\left(2+2\left(\frac{2p}{1+p}\right)^i-4\left(\frac{2p}{1+p}\right)^s+\left(\frac{2p}{1+p}\right)^{2s-i} \right) \notag \\
 &\qquad \quad +2s\log \frac{2}{1+p} - \log \left(\left(2-\left(\frac{2p}{1+p}\right)^s\right)^2\right)\Bigr]\Bigr)\notag\\
 \end{align*}
(the second and third term in the exponent coming from $\log b_i$ and the last two terms coming from $\log b_0=\log \left(q(s,d)^2\right)$)
\begin{align}
&\leq \frac{2}{d^{i \varepsilon}}\exp\Bigl(d\Bigl[\log\left(2+2\left(\frac{2p}{1+p}\right)^i-4\left(\frac{2p}{1+p}\right)^s+\left(\frac{2p}{1+p}\right)^{2s-i}\right)\notag\\
&\qquad \qquad-\left(2- \frac{i}{s}\right)\log\left(2-\left(\frac{2p}{1+p}\right)^s\right)\Bigr]\Bigr), \label{big scary expression}
\end{align}
with the inequality in the last line coming from substituting $t_{K_s}(p)+ \frac{2\log s}{d} +\varepsilon \frac{\log d}{d}$ for $x$, and rearranging terms. We now resort to the following technical lemma.
\begin{lemma}\label{technical lemma}
For all $y \in [0,1]$ and all integers $0\leq i \leq s$, the following inequality holds:
\begin{align*}\left(2+2y^i-4y^s+y^{2s-i}\right)^s \leq \left(2-y^s\right)^{2s-i}. \end{align*}
\end{lemma}
We defer the proof of Lemma~\ref{technical lemma} (which is a simple albeit lengthy exercise) to Appendix A. Set $y = \left(\frac{2p}{1+p}\right)$. As $0 < p < 1$, we have $y \in (0,1)$. Applying Lemma~\ref{technical lemma}, we have
\[\log \left( 2+2y^i-4y^s+y^{2s-i}\right) - \left(2-\frac{i}{s}\right) \log \left(2-y^s\right)<0.\]
Substituting this into the expression inside the exponential in (\ref{big scary expression}), we get
\[\frac{\binom{n-s}{s-i}\binom{s}{i}}{\binom{n-s}{s}}\frac{b_i}{b_0}\leq \frac{2}{d^{i \varepsilon}}.\]
Thus
\begin{align*}
\mathbb{E}\left(X^2\right)&\leq \binom{n}{s}\binom{n-s}{s}b_0 \left(1 + \frac{2}{d^{\varepsilon}}\left(1+ \frac{1}{d^{\varepsilon}}+\frac{1}{d^{2\varepsilon}}+\cdots \right)\right)\\
&=\left(\mathbb{E}X\right)^2(1+o(1)).
\end{align*}
In particular, $\textrm{Var}(X)=o\left(\left(\mathbb{E}X\right)^2\right)$, and by Chebyshev's inequality whp $X$ is at least $\frac{1}{2}\mathbb{E}X>0$. Thus whp $G_{[n],d,p}$ contains (many) $s$-cliques.
\end{proof}

\begin{remark}
We have shown that the transition between whp no $s$-cliques and whp many $s$-cliques in $G_{[n],d,p}$ takes place inside a window of width (with respect to $x$) of order $O\left(\frac{\log d}{d}\right)$. In the case when $s$ is bounded, $s(d)=O(1)$, it is easy to run through the proof of Theorem~\ref{binomial s<<d/log d} again and show that in fact we may replace $\varepsilon \log d$ in the statement of the Theorem by any function $g=g(d)$ tending to infinity with $d$, so that the width of the window may be reduced to $O\left(\frac{g}{d}\right)$.
\end{remark}

Having proven Theorem~\ref{binomial s<<d/log d}, we now turn our attention to the following question. Let $s \in \mathbb{N}$ be fixed. What is the dimension of the features subcubes in the first $s$-clique to appear in $G_{[n],d, p}$?

As always, write $x$ for $\frac{1}{d}\log n$. Let $S_{\alpha}$ be a subcube of dimension $\alpha d$. Suppose $S_{\alpha}$ is the feature subcube of some $v \in [n]$, $f(v)=S_{\alpha}$.  Then the expected number of $s$-cliques involving $v$ is
\begin{align*}
\mathbb{E}&\#\{K_{s-1} \textrm{ meeting } S_{\alpha}\}=\binom{n-1}{s-1} \left(\frac{1+p}{2}\right)^{(s-1)(1-\alpha)d}\left(2\left(\frac{1+p}{2}\right)^{s-1}-p^{s-1}\right)^{\alpha d}\\
%
\end{align*}
An application of Wald's equation yields that the expected number $E_{\alpha}^s$ of pairs $(v, \mathcal{S})$ for which
 (i)  $v\in [n]$ is a vertex with a feature subcube $f(v)$ of dimension $\alpha d$, and
(ii) $\mathcal{S}$ is an $s$-set of vertices from $[n]$ containing $v$ and inducing an $s$-clique in $G_{[n], d, p}$,
 is:
\begin{align*}
E_{\alpha}^s&=\mathbb{E}\#\{\textrm{$\alpha d$-dimensional feature subcubes}\} \times \mathbb{E}\#\{K_{s-1} \textrm{ meeting } S_{\alpha}\}\\
&=n \binom{d}{\alpha d}p^{\alpha d}(1-p)^{(1-\alpha)d}\binom{n-1}{s-1} \left(\frac{1+p}{2}\right)^{(s-1)(1-\alpha)d}\left(2\left(\frac{1+p}{2}\right)^{s-1}-p^{s-1}\right)^{\alpha d}\\
&=\exp\Bigl(d \Bigl[ sx +\alpha \log \left(\left(\frac{p}{\alpha}\right)\left(2\left(\frac{1+p}{2}\right)^{s-1}-p^{s-1}\right)\right)\\
& \qquad + (1-\alpha)\log\left( \left(\frac{1-p}{1-\alpha}\right)\left(\frac{1+p}{2}\right)^{s-1}\right) \Bigr]+o(d)\Bigr).
\end{align*}
Define
\begin{align*}t^{\alpha}_{K_s}:&=-\frac{1}{s}\Bigl(\alpha \log \left(\frac{p}{\alpha}\cdot\left(2\left(\frac{1+p}{2}\right)^{s-1}-p^{s-1}\right)\right)\\
& \qquad + (1-\alpha)\log\left( \left(\frac{1-p}{1-\alpha}\right)\cdot\left(\frac{1+p}{2}\right)^{s-1}\right)\Bigr).
\end{align*}
The expression above can then be rewritten as $E_{\alpha}^s= e^{sd\left(x -t^{\alpha}_{K_s}+o(1)\right)}$. Set 
\[\alpha_s=\alpha_s(p):= p\left(\frac{2\left(\frac{1+p}{2}\right)^{s-1}-p^{s-1}}{2\left(\frac{1+p}{2}\right)^{s}-p^{s}}\right).\] 
\begin{remark}
The quantity $\alpha_s$ is exactly the probability that a given vertex receives colour $\star$ in $G_{[n],1,p}$ conditional on it forming an $s$-clique with a fixed $(s-1)$-set of vertices. In particular it follows from Remark~\ref{binomial d as intersection of binomial 1} that $\alpha_s d$ is the expected dimension of feature subcubes in an $s$-clique in $G_{[n],d,p}$.
\end{remark}
\begin{remark}\label{alphas monotonicity}
For $0<p<1$ fixed, the sequence $\left(\alpha_s\right)_{s \in \mathbb{N}}$ is strictly increasing and tends to $\frac{2p}{1+p}$ as $s\rightarrow \infty$. Note in particular that for all $s>1$, $\alpha_s(p)>\alpha_1(p)=p$.
\end{remark}

\begin{proposition}\label{alpha-s unique minimum of t-alpha-Ks}
Let $p\in(0,1)$ be fixed. Then for every $s \in \mathbb{N}$ the following equality holds:
\[t_{K_s}=t^{\alpha_s}_{K_s}.\]
Moreover, as a function of $\alpha$, $t^{\alpha}_{K_s}$ is strictly decreasing for $\alpha \in [0, \alpha_s)$ and strictly increasing for $\alpha \in (\alpha_s,1]$. In particular, $\alpha_s$ is the unique minimum of $t^{\alpha}_{K_s}$ over all $\alpha \in [0,1]$.
\end{proposition}
\begin{proof}
The first part of our proposition is a simple calculation. Recall from the proof of Theorem~\ref{binomial s<<d/log d} that $q(s,1)= 2\left(\frac{1+p}{2}\right)^s-p^s$ is the probability that a given $s$-set of vertices forms an $s$-clique in $G_{[n],d, 1}$. Note that 
\begin{align*}
&t_{K_s}=-\frac{1}{s}\log q(s,1), \qquad \alpha_s = \frac{p q(s-1,1)}{q(s,1)}\  \textrm{ and}\\
&q(s,1)-pq(s-1,1)=(1-p)\left(\frac{1+p}{2}\right)^{s-1}.
\end{align*}
Thus,
\begin{align*}
t^{\alpha_s}_{K_s}&=-\frac{1}{s}\Bigl(\alpha_s \log \left(\left(\frac{p}{\alpha_s}\right)\left(2\left(\frac{1+p}{2}\right)^{s-1}-p^{s-1}\right)\right)+ (1-\alpha_s)\log\left(\left(\frac{1-p}{1-\alpha_s}\right)\left(\frac{1+p}{2}\right)^{s-1}\right)\Bigr)\\
&=-\frac{1}{s}\left(\alpha_s \log \left(\frac{q(s,1)}{q(s-1,1)}q(s-1,1)\right)+ (1-\alpha_s)\log\left( \frac{(1-p)q(s,1)}{q(s,1)-pq(s-1,1)}\left(\frac{1+p}{2}\right)^{s-1}\right)\right)\\
&=-\frac{1}{s}\left(\alpha_s \log q(s,1)+ (1-\alpha_s)\log q(s,1)\right)=-\frac{1}{s}\log q(s,1)=t_{K_s},
\end{align*}
as required.

Now, let us show that $t^{\alpha_s}_{K_s}$ is in fact the unique minimum of $t^{\alpha}_{K_s}$ over $\alpha \in [0,1]$. 
Making use of our observations above, we may write $st^{\alpha}_{K_s}$ as
\[s t^{\alpha}_{K_s}= \alpha \log \left(\frac{\alpha}{q(s,1)\alpha_s}\right) +(1-\alpha)\log \left(\frac{1-\alpha}{q(s,1)(1-\alpha_s)}\right).\]
The derivative with respect to $\alpha$ is
\[s \frac{d}{d\alpha}\left(t^{\alpha}_{K_s}\right)=\log \left(\frac{\alpha}{q(s,1)\alpha_s}\right) - \log\left(\frac{1-\alpha}{q(s,1)(1-\alpha_s)} \right),\]
which is strictly negative for $0\leq \alpha<\alpha_s$, zero for $\alpha=\alpha_s$ and strictly positive for $1\geq \alpha>\alpha_s$, establishing our claim.
\end{proof}
Used in conjunction with Theorem~\ref{binomial s<<d/log d} (or more precisely Corollary~\ref{binomial s constant clique threshold}), Proposition~\ref{alpha-s unique minimum of t-alpha-Ks} enables us to identify with quite some precision the dimension of the feature subcubes of the vertices which witness the emergence of $s$-cliques in $G_{[n],d,p}$. Formally we return to our dynamic view of the model, and we consider the graph $B_n=G_{[n],d, p}$ at the hitting time $n= N^b_{s}$ for the property of containing a clique on $s$ vertices. By definition of the hitting time, $G_{[n],d,p}$ contains at least one $K_s$-subgraph. Set $W_{s}= W_{s}(d,p)$ to be the set of all vertices in $[n]$ which are contained in such a $K_s$-subgraph.
\begin{proposition}\label{prop: dimension of subcubes in s-cliques}
Whp, all feature subcubes of vertices contained in $W_{s}(d,p)$ have dimension $\left(\alpha_s +o(1)\right)d$.
\end{proposition}
\begin{proof}
Fix $\varepsilon>0$. By Proposition~\ref{alpha-s unique minimum of t-alpha-Ks}, there exists $\delta>0$ such that if $t^{\alpha}_{K_s}\leq t_{K_s} +\delta = t^{\alpha_s}_{K_s} +\delta$, then $\vert \alpha -\alpha_s \vert <\varepsilon$. 
By Corollary~\ref{binomial s constant clique threshold}, whp the hitting time $N^b_{s}$ for containing an $s$-clique satisfies $e^{t_{K_s}d - \frac{\delta}{2}d}  \leq  N^b_{s} \leq e^{t_{K_s}d + \frac{\delta}{2}d}$. We show that for $\vert V\vert =e^{xd}$ and $\vert x-t_{K_s} \vert <\frac{\delta}{2}$ whp no vertex in $G_{V,d,p}$ with a feature subcube of dimension $\alpha d$ with $\vert \alpha_s -\alpha \vert \geq \varepsilon$ is contained in a copy of $K_s$. Since $\varepsilon>0$ was arbitrary, this is enough to establish the proposition.

Set $I_{\varepsilon}=\left\{\frac{i}{d}: \ i \in \{0,1,\ldots d \}\right\} \setminus (\alpha_s-\varepsilon, \alpha_s+\varepsilon)$. 
The expected number of pairs $(v, S)$ where $v\in V$ has a feature subcube of dimension $\alpha d$ for some $\alpha \in I_{\varepsilon}$, $v \in S$ and $S\subseteq V$ induces a copy of $K_s$ in $G_{V,d,p}$ is:
\begin{align*}
\sum_{\alpha \in I_{\varepsilon}} E^s_{\alpha}& =\sum_{\alpha \in I_{\varepsilon}} e^{sd \left(x - t^{\alpha}_{K_s}\right)}\leq d e^{sd\left(x -t_{K_s}-\delta+o(1)\right)}\leq e^{-\frac{s\delta}{2}d+o(d) }=o(1). 
\end{align*}
Markov's inequality thus implies that whp no such pair $(v, S)$ exists in $G_{V,d,p}$. In particular all vertices of $G_{V,d,p}$ which are contained in a copy of $K_s$ must have dimension $\alpha d$ for some $\alpha: \ \vert \alpha -\alpha_s \vert <\varepsilon$, as claimed.
\end{proof}

\subsection{The covering threshold}\label{binomial: at covering}
We may view the question of covering the hypercube $Q_d$ with randomly selected subcubes as an instance of 
the following problem.
\begin{problem}[Generalised Coupon Collector Problem]\label{generalised coupon problem}
	Let $\Omega$ be a (large) finite set, and let $X$ be a random variable taking values in the subsets of $\Omega$. Suppose we are given a 
	sequence of independent random variables $X_1, X_2, \ldots, X_n$ with distribution given by $X$. When (for which values of $n$) do we have 
	$\bigcup_{i=1}^n X_i=\Omega$ holding whp?
\end{problem}
When $X$ is obtained by selecting a singleton from $\Omega$ uniformly at random, Problem~\ref{generalised coupon problem} is the classical coupon collector problem (see~\cite{Newman60} for an early incarnation of the problem).
\begin{proposition}\label{crude collector bounds}
Set $\vert \Omega \vert =m$. Suppose $X$ is such that for every $\omega, \omega' \in \Omega$, $\mathbb{P}(\omega \in X)=\mathbb{P}(\omega' \in X)$. Then, for every fixed $\varepsilon>0$,
\[\lim_{m\rightarrow \infty}\mathbb{P}\left(\bigcup_{i=1}^n X_i = \Omega\right)=\left\{\begin{array}{ll}
0 & \textrm{if }n \ll \frac{m}{\mathbb{E} \vert X \vert}\\
1 & \textrm{if }n \geq (1+\varepsilon)\frac{m \log m}{\mathbb{E} \vert X \vert}\end{array}\right.\]
\end{proposition}
\begin{proof}
For any fixed $\omega \in \Omega$,
\begin{align*}
\mathbb{P}\left(\omega \notin \bigcup_{i=1}^n X_i\right)&=\left(1-\mathbb{P}(\omega \in X)\right)^n=\left( 1- \frac{\mathbb{E}\vert X \vert }{m}\right)^n.
\end{align*}
Thus if $n =o\left(\frac{m}{\mathbb{E} \vert X \vert }\right)$, the probability that $\omega \notin \bigcup_{i=1}^n X_i$ is $e^{-o(1)}=1-o(1)$, whence whp $\bigcup_{i=1}^n X_i \neq \Omega$. On the other hand, if $n\geq(1+\varepsilon)\frac{m \log m}{\mathbb{E} \vert X \vert}$, then
\begin{align*}
\mathbb{E} \left\vert \Omega \setminus \bigcup_{i=1}^m X_i \right\vert& =\sum_{\omega \in \Omega}\mathbb{P}\left(\omega \notin \bigcup_{i=1}^n X_i\right)=m \left( 1- \frac{\mathbb{E}\vert X \vert }{m}\right)^n=o(1)
\end{align*}
so that by Markov's inequality whp there are no uncovered elements and $\bigcup_{i=1}^n X_i=\Omega$.
\end{proof}
The bounds we give in Proposition~\ref{crude collector bounds} are very crude, but are essentially best possible (see~\cite{Aldous91, FalgasRavryLarssonMarkstrom15}). 
In our setting, we have $\Omega =Q_d$, and the $(X_i)_{i=1}^n$ are the feature subcubes of vertices in the binomial random subcube intersection graph $G_{[n],d, p}$. Note that the expected volume of a feature subcube $f(v)$ is:
\begin{align*}
\mathbb{E}\vert f(v)\vert &= \sum_{i=0}^d \mathbb{P}(f(v) \textrm{ has dimension }i)2^i=\sum_{i=0}^d \binom{d}{i}(1-p)^{d-i}(2p)^i=(1+p)^d,
\end{align*}
while on the other hand typical feature subcubes have dimension $pd+o(d)$ and thus volume $2^{pd+o(d)}$. Since $2^p<1+p$ for all $p \in (0,1)$, typical feature subcubes have a volume much smaller than the expected volume. In particular, the variance of the volume of a feature subcube is large, and our covering problem differs significantly from the classical coupon collector problem.

We need to make one more definition before proceeding to the proof of Theorem~\ref{binomial covering threshold}.
\begin{definition}
The \emph{Hamming distance} $\textrm{dist}(y,y')$ between two elements $y,y'$ of $Q_d$ is the number of coordinates in which they differ.
\end{definition}
\begin{proof}[Proof of Theorem~\ref{binomial covering threshold}]
Without loss of generality, we may assume that $V=[n]$, and that we are working with the binomial random subcube intersection graph $G_{[n],d, p}$. Let $\mathbf{0}$ denote the all zero element $(0,0, \ldots 0)$ from $Q_d$. The expected number of elements of $Q_d=\{0,1\}^d$ not covered by the union of the feature subcubes $\bigcup_{v=1}^n f(v)$ is
\begin{align*}
\mathbb{E}\vert Q_d\setminus \bigcup_{v=1}^n f(v) \vert &= \vert Q_d \vert \mathbb{P}\left(\mathbf{0} \notin \bigcup_{v=1}^n f(v) \right)=2^d \left(1- \left(\frac{1+p}{2}\right)^d\right)^n\\
&=\exp\left(d \log 2 - n\left(\frac{1+p}{2}\right)^d\left(1+O\left(\frac{1+p}{2}\right)^d\right)\right)\\
\end{align*}

Now let $\varepsilon$ be fixed with $0< \varepsilon < \log 2$.

\noindent \textbf{Upper bound:} suppose $n=e^{xd}\geq \left(\frac{2}{1+p}\right)^dd\left(\log 2 + \varepsilon\right)$.
Then the expected number of uncovered elements of $Q_d$ is at most $e^{-\varepsilon d +o(d)}=o(1)$, whence we deduce from Markov's inequality that whp $\bigcup_{v=1}^n f(v)= Q_d$.

\noindent \textbf{Lower bound:} suppose $n=e^{xd}= \lfloor \left(\frac{2}{1+p}\right)^dd\left(\log 2 - \varepsilon\right)\rfloor$.
Then the expected number of uncovered elements of $Q_d$ is $e^{\varepsilon d+o(d)}$, which is large. We use Chebyshev's inequality to show the actual number of uncovered elements is concentrated about this value.

For $1\leq i \leq d$, let $\mathbf{e}_{[i]}$ denote the element of $Q_d=\{0,1\}^d$ with its first $i$ coordinates equal to $1$ and its last $d-i$ coordinates equal to $0$. Clearly we have $\textrm{dist}(\mathbf{0}, \mathbf{e}_{[i]})=i$. The probability that neither of $\mathbf{0}, \mathbf{e}_{[i]}$ is covered is:
\begin{align*}
\mathbb{P}\left(\mathbf{0}, \mathbf{e}_{[i]} \notin \bigcup_{v=1}^nf(v)\right)&= \left(1- 2\left(\frac{1+p}{2}\right)^d + p^i\left(\frac{1+p}{2}\right)^{d-i}\right)^n\\
&= \exp\left(-n\left[2\left(\frac{1+p}{2}\right)^d - p^i\left(\frac{1+p}{2}\right)^{d-i}\right]+ O\left(n \left(\frac{1+p}{2}\right)^{2d}\right)\right)\\
&= \exp\left(-\left(\log 2 - \varepsilon\right)d\left(2- \left(\frac{2p}{1+p}\right)^i\right)+
o(1)\right).
\end{align*}
Thus
\begin{align}
\mathbb{E} \left(\left\vert Q_d \setminus \bigcup_{v=1}^n f(v)\right\vert ^2\right)
&= 2^d \sum_{i=0}^d \binom{d}{i}\mathbb{P}\left(\mathbf{0}, \mathbf{e}_{[i]} \notin \bigcup_{v=1}^nf(v)\right) \notag\\
&=e^{2 \varepsilon d} \sum_{i=0}^d \frac{\binom{d}{i}}{2^d} \exp\left( \left(\frac{2p}{1+p}\right)^i\left(\log 2- \varepsilon\right) d+o(1)\right).\label{square of expected number of uncovered vertices}
\end{align}
Pick $\eta: \ 0 < \eta < 1/2$ sufficiently small such that 
\[\varepsilon> \eta\log\frac{1}{\eta} +(1-\eta)\log \frac{1}{1-\eta}\] is satisfied. Then
\begin{align}
\sum_{0 \leq i \leq \eta d} \frac{\binom{d}{i}}{2^d}\exp\left( \left(\frac{2p}{1+p}\right)^i\left(\log 2- \varepsilon \right)d\right)&< \eta d \frac{\binom{d}{\eta d}}{2^d} \exp\left(\left(\log 2-\varepsilon \right) d\right)\notag \\
&=\eta d \binom{d}{\eta d}e^{-\varepsilon d}=o(1). \label{uncovered vertices: small hamming distance}
\end{align}   
On the other hand, for $i > \eta d$ we have 
$\left(\frac{2p}{1+p}\right)^i\left(\log 2- \varepsilon\right) d =o(1), 
$
since $\frac{2p}{1+p}<1$, so that 
\begin{align}
\sum_{i\geq \eta d} &\frac{\binom{d}{i}}{2^d}\exp\left( \left(\frac{2p}{1+p}\right)^i\left(\log 2- \varepsilon\right) d+o(1)\right)\notag \\
&\leq \sum_{i\geq \eta d}\frac{\binom{d}{i}}{2^d}e^{o(1)}\leq 1+o(1). \label{uncovered vertices: large hamming distance}
\end{align}
Substituting the bounds (\ref{uncovered vertices: small hamming distance}) and (\ref{uncovered vertices: large hamming distance}) into (\ref{square of expected number of uncovered vertices}), we get
\begin{align*}
\mathbb{E}\left(\left\vert Q_d \setminus \bigcup_{v=1}^n f(v)\right\vert ^2\right) &\leq e^{2 \varepsilon d} (1+o(1))= \left(1+o(1)\right)\left(\mathbb{E}\left\vert Q_d \setminus \bigcup_{v=1}^n f(v)\right\vert\right)^2,
\end{align*}
whence $\textrm{Var} \vert Q_d \setminus \bigcup_{v=1}^n f(v)\vert = o\left(\mathbb{E}\left\vert Q_d \setminus \bigcup_{v=1}^n f(v)\right\vert^2 \right)$. It follows by Chebyshev's inequality that whp  
$\bigcup_{v=1}^n f(v)$ leaves $(1+o(1))e^{\varepsilon d}$ elements of $Q_d$ uncovered when $n \leq \left(\frac{2}{1+p}\right)^dd\left(\log 2 - \varepsilon\right)$, as claimed.
\end{proof}

\subsection{Above the covering threshold}\label{binomial: above covering}
\begin{proof}[Proof of Theorem~\ref{binomial s>>d/logd}]
Without loss of generality, we may assume that $V=[n]$. Fix $\varepsilon>0$, and let $s=s(d)$ be a sequence of natural numbers with $\frac{s \log d}{d}\rightarrow \infty$ as $d\rightarrow \infty$.

\noindent \textbf{Upper bound:}
Here, unlike in the proof of Theorem~\ref{binomial s<<d/log d}, we eschew estimates of the total number of $s$-cliques present in $G_{v,d,p}$, but proceed instead via  a covering argument. Indeed, by the Helly property, $G_{[n],d,p}$ contains an $s$-clique if and only if some element of the ambient hypercube $Q_d$ is contained in at least $s$ feature subcubes. Denote by 
\[\mathrm{Vol}[n]:=\sum_{v=1}^n \vert f(v)\vert\]
the sum of the sizes of the feature subcubes. By linearity of expectation,
\begin{align*}
\mathbb{E}\mathrm{Vol}[n]&=n \mathbb{E}\vert f(1)\vert =n(1+p)^d.
\end{align*}
Set $x= \log \frac{2}{1+p} + \frac{\log s}{d}+ \frac{\varepsilon}{d}$. For $n\geq\lceil e^{xd}\rceil$, we have $\mathbb{E} \mathrm{Vol}[n]\geq e^{\varepsilon}s 2^d$, which means that elements of the ambient hypercube are expected to be contained in $e^{\varepsilon}s>s$ feature subcubes. Thus, to show that $G_{[n],d, p}$ whp contains (many) $s$-cliques for this value of $n$, it is enough to show that $\mathrm{Vol}[n]$ is concentrated about its mean. Again, we use the second-moment method to do this. By linearity of variance we have
\begin{align*}
\textrm{Var}\mathrm{Vol}[n]=n \textrm{Var}(f(1))&=n\left(\left(\sum_{i=1}^d \binom{d}{i} (1-p)^{d-i}p^i 2^{2i} \right) -(1+p)^{2d} \right)\\
&=n\left((1+3p)^d- (1+2p+p^2)^d\right).
\end{align*}
Applying Chebyshev's inequality,
\begin{align*}
\mathbb{P}\left(\mathrm{Vol}[n]< s 2^d\right) &=\mathbb{P}\left(\mathrm{Vol}[n]< e^{-\varepsilon}\mathbb{E}\mathrm{Vol}[n]\right)\\
&\leq \frac{\textrm{Var} \mathrm{Vol}[n]}{\left(1-e^{-\varepsilon}\right)^2 \left(\mathbb{E}\textrm{Vol}[n]\right)^2}\\
&< \frac{(1+3p)^d}{\left(1-e^{-\varepsilon}\right)^2 n (1+p)^{2d}}\\
&\leq\frac{1}{\left(1-e^{-\varepsilon}\right)^2} \frac{(1+3p)^d}{2^d(1+p)^d}\qquad \textrm{(substituting in the value of $n$)}\\
&= \frac{1}{\left(1-e^{-\varepsilon}\right)^2} \left(1- \left(\frac{1-p}{2(1+p)}\right)\right)^d =o(1).
\end{align*}
In particular,
\begin{align*}
\mathbb{P}\left(G_{[n,d,p]} \textrm{ contains an $s$-clique}\right)&\geq 1-\mathbb{P}\left(\mathrm{Vol}[n] < s2^d\right)=1-o(1),
\end{align*}
proving the claimed upper bound on the threshold for the emergence of $s$-cliques.
\begin{remark}
The proof above actually shows a little more: for $x> \log \left(\frac{1+3p}{(1+p)^2}\right)$,  $\textrm{Var}\mathrm{Vol}[n]=o\left(\left(\mathbb{E} \mathrm{Vol}[n]\right)^2 \right)$ and thus by Chebyshev's inequality whp $\mathrm{Vol}[n]=(1+o(1))n(1+p)^d$. In other words there are sufficiently many feature subcubes at this point that the large variance of their individual volumes ceases to matter. Note that this occurs before the covering threshold, since $\log \left(\frac{1+3p}{(1+p)^2}\right)< \log \frac{2}{1+p}$.
\end{remark}

\noindent \textbf{Lower bound when $s= O(d)$:} in this case we use Markov's inequality just as in the proof of Theorem~\ref{binomial s<<d/log d}. Set $x= \log\frac{2}{1+p}+ \frac{\log s}{d}- \varepsilon\frac{ \log d}{d}$, and let $n=\lfloor e^{xd} \rfloor$.
 Let $X=X(d)$ be the number of $s$-cliques in the graph. Then
\begin{align*}
\mathbb{E} &X=\binom{n}{s}\left(2\left(\frac{1+p}{2}\right)^s-p^s\right)^d\\
&= \exp\left( s xd - s \log s - ds\log \frac{2}{1+p} + d\log\left(2- \left(\frac{2p}{1+p}\right)^s\right)+O\left(s\right)  \right)\\
&= \exp\left(-s\varepsilon \log d +O\left(\max(s,d) \right)\right)
=o(1) \qquad \textrm{(since $s\log d\gg \max(s,d)$)}
\end{align*}
so that whp $X=0$ and $G_{[n],d,p}$ contains no $s$-clique.

\noindent \textbf{Lower bound when $s\gg d$:} here we use a covering idea. Suppose $n = e^{-\varepsilon}s \left(\frac{2}{1+p}\right)^d$. 
The number $C_{\mathbf{0}}$ of feature subcubes containing the element $\mathbf{0}=(0, 0\ldots 0)$ is the sum of $n$ independent identically distributed Bernoulli random variables with parameter $\left(\frac{1+p}{2}\right)^d$. We have $\mathbb{E} C_{\mathbf{0}}= n\left(\frac{1+p}{2}\right)^d \leq e^{-\varepsilon}s$. Applying a Chernoff bound, we deduce that
\begin{align*}
\mathbb{P}(C_{\mathbf{0}}\geq s)&\leq e^{-\frac{\varepsilon^2}{3}s}.
\end{align*}
In particular the expected number of elements of $Q_d$ contained in at least $s$ feature subcubes 
is $\vert Q_d\vert \mathbb{P}(C_{\mathbf{0}}\geq s) \leq 2^de^{-\frac{\varepsilon^2}{3}s}$, which is $o(1)$ for $s\gg d$. It follows by Markov's inequality that whp there is no such element, and thus, by the Helly property for subcube intersections, that whp $G_{[n],d, p}$ contains no copy of $K_s$. Further by monotonicity of the property of containing an $s$-clique, whp $G_{[n'],d ,p}$ fails to contain a $K_s$ for any $n'\leq n$.
\end{proof}

\section{The uniform model}\label{uniformsection}
In this section, we prove our results for the uniform model. We note that these are generally less precise than those we obtained for the binomial model, owing to the greater difficulty of performing clique computations.

\subsection{Summary}
Fix $s\in\mathbb{N}$. We established in Section~\ref{binomialsection} (Proposition~\ref{prop: dimension of subcubes in s-cliques}) that in $G_{V,d,p}$, whp the feature subcubes of the vertices in the first $s$-cliques to appear as we increase $\vert V\vert$ all have dimension $(\alpha_s +o(1))d$, where $\alpha_s$ is the function:
\[\alpha_s: \ p \mapsto \frac{p\left(2\left(\frac{1+p}{2}\right)^{s-1}-p^{s-1}\right)}{\left(2\left(\frac{1+p}{2}\right)^s-p^s\right)}.\]
We show in Proposition~\ref{alphas inverse} that $\alpha_s$ is a bijection from $(0,1)$ to itself. This will allow us to determine the threshold for the appearance of $s$-cliques in the uniform model.

\begin{theorem}\label{uniform: below the threshold}
Let $\alpha \in (0,1)$ and $s\in \mathbb{N}$ be fixed, and let $k(d)=\lfloor \alpha d \rfloor$. Set $p={\alpha_s}^{-1}(\alpha)$. Then, the threshold for the appearance of $s$-cliques in $G_{V,d, k}$ is
\[T_{K_s}(\alpha)=t_{K_s}(p)+ \alpha \log \frac{p}{\alpha}+(1-\alpha)\log\frac{1-p}{1-\alpha}.\]
\end{theorem}

\begin{theorem}\label{uniform: at the threshold}
Let $\alpha \in (0,1)$ and $\varepsilon>0$ be fixed, and let $k(d)=\lfloor \alpha d \rfloor$.  Let $V=V(d)$ be a sequence of vertex sets with $\vert V(d)\vert = e^{xd}$. Then, for the uniform model $G_{V,d,k}$,
\[\lim_{d\rightarrow \infty} \mathbb{P}\left(\bigcup_{v \in V}f(v)=Q_d\right)=\left\{\begin{array}{ll}
0 & \textrm{if $x(d)\leq (1-\alpha)\log 2 + \frac{\log d}{d} + \frac{\log \left(\log 2 -\varepsilon\right)}{d}$} \\
1& \textrm{if $x(d)\geq (1-\alpha)\log 2 + \frac{\log d}{d} +\frac{\log \left(\log 2 +\varepsilon\right)}{d}$.}
\end{array}\right.\]
\end{theorem}
\begin{corollary}\label{uniform: covering threshold}
Let $\alpha \in (0,1)$ be fixed, and let $k(d)=\lfloor \alpha d \rfloor$. Then the threshold for covering the ambient hypercube $Q_d$ with the feature subcubes from $G_{V,d,k}$ is
\[T_{\textrm{cover}}(\alpha)=(1-\alpha) \log 2.\]
\end{corollary}
\begin{remark}
As we observed in Remark~\ref{alphas monotonicity}, we have $\lim_{s\rightarrow \infty}\alpha_s(p)=\frac{2p}{1+p}$. From this we deduce that for large $s$, we have $\alpha_s^{-1}(\alpha)=\frac{\alpha}{2-\alpha}+o(1)$. Substituting this into $T_{K_s}(\alpha)$, we see that 
\[T_{K_s}(\alpha)\rightarrow T_{\textrm{cover}}(\alpha)\]
as $s\rightarrow \infty$, mirroring our observation in Remark~\ref{covering as limit of clique thresholds} for the binomial model.
\end{remark}

\begin{theorem}\label{uniform: above the threshold}
Let $\alpha \in (0,1)$ and $\varepsilon>0$ be fixed, and let $k(d)=\lfloor \alpha d \rfloor$. Let $s=s(d)$ be a sequence of natural numbers with $\frac{s}{d}\rightarrow \infty$ as $d\rightarrow \infty$. Suppose $V=V(d)$ is a sequence of vertex sets. Then,
\[\lim_{d\rightarrow \infty} \mathbb{P}\left(G_{V,d,k} \textrm{ contains an $s$-clique} \right)=\left\{\begin{array}{ll}
0 & \textrm{if $\vert V(d)\vert \leq (1-\varepsilon)s 2^{d-k}$} \\
1& \textrm{if $\vert V(d)\vert \geq (s-1)2^{d-k}+1$}.
\end{array}\right.\]
\end{theorem}
\begin{remark}
Theorem~\ref{uniform: below the threshold} and Corollary~\ref{uniform: covering threshold} show how significant `outliers' (subcubes with unusually high dimension) are for the behaviour of the binomial model. Indeed, Proposition~\ref{prop: dimension of subcubes in s-cliques} tells us that for $0<p<1$ fixed and $s\geq 3$, the vertices in the first $s$-clique to appear in $G_{V,d,p}$ have feature subcubes of dimension $\left( \alpha_s(p)+o(1)\right) d$. Since $\alpha_s(p)>p$ it shall follow straightforwardly from the proof of Theorem~\ref{uniform: below the threshold} that $t_{K_s}(p)<T_{K_s}(p)$. Similarly, by Corollaries~\ref{binomial covering threshold} and~\ref{uniform: covering threshold}, we have for $0<p<1$ fixed that 
\[t_{\textrm{cover}}(p)=\log\frac{2}{1+p}<(1-p)\log 2 =T_{\textrm{cover}}(p).\]

From the covering threshold upwards, Corollary~\ref{uniform: covering threshold} and Theorem~\ref{uniform: above the threshold} suggest that, when considering  questions about cliques and covering, the right instance of the binomial model to compare $G_{V,d,\lfloor \alpha d\rfloor}$ with is $G_{V,d, p}$ with $p=2^{\alpha}-1$ (rather than $p=\alpha$ as we might have expected). For these two models, the covering threshold and the thresholds for higher order cliques coincide. Since both models have the same expected volume of feature subcubes, this vindicates the use of volume/covering arguments for determining the thresholds for higher order cliques. Note however that $G_{V,d, \lfloor \alpha d \rfloor}$ and $G_{V,d, 2^{\alpha}-1}$ have different thresholds for lower order cliques. 
Our binomial model and uniform model thus behave differently, and there is no good coupling between them below the covering threshold.
\end{remark}

Finally, let us add that, just as in the binomial model, the clique number $\omega(G_{V,d,k})$ in the uniform model undergoes a transition around the covering threshold. 
\begin{corollary}
Let $\alpha\in (0,1)$ be fixed and let $k=k(d)=\lfloor \alpha d \rfloor$. Let $V(d)$ be a sequence of vertex-sets and $x(d)= \frac{1}{d} \log \vert V(d)\vert$ as usual. The following hold:
\begin{itemize} 
\item if there is $s \in \mathbb{N}$ and $\varepsilon>0$  such that $T_{K_s}+\varepsilon<x <T_{K_{s+1}}$,
then whp $\omega(G_{V,d,k})=s$;
\item if there is $s \in \mathbb{N}$ such that  
 $x=T_{K_s}+o(1)$,
 then whp $\omega(G_{V,d,k})\in \{s, s-1\}$;
\item if there is $\gamma>1$ such that  
 $x=x(d)=T_{\textrm{cover}}+ \gamma \frac{\log d}{d}+o\left(\frac{\log d}{d}\right)$,
 then whp $\omega(G_{V,d,k})$ has order $d^{\gamma+o(1)}$;
\item if there is $c>0$ such that  
 $x=x(d)=T_{\textrm{cover}}+c+o(1)$, 
 then whp $\omega(G_{V,d,k})$ has order $e^{cd+o(d)}$.
\end{itemize}
\end{corollary}
\begin{remark}
There is a gap here: we do not know what the order of the clique number is when $x=x(d)=T_{\textrm{cover}}+\gamma\frac{\log d}{d}+o(\frac{\log d}{d})$ for a fixed real $\gamma$ with $0<\gamma \leq 1$. We make the natural conjecture that for this value of $x(d)$, we should have $\omega(G_{V,d,k})=d^{\gamma+o(1)}$, similarly to the binomial model.
\end{remark}
Theorems~\ref{uniform: below the threshold}, \ref{uniform: at the threshold} and \ref{uniform: above the threshold} are proved in Sections~\ref{subsection: uniform model below covering}, \ref{subsection: uniform model at covering} and~\ref{subsection: uniform model above covering} respectively. Our results give whp lower and upper bounds on certain hitting times for the uniform model, and we often split their proofs accordingly into two parts.

\subsection{Below the covering threshold}\label{subsection: uniform model below covering}

\begin{proposition}\label{alphas inverse}
The function $\alpha_s$ is a bijection from $[0,1]$ to itself, and has a continuous inverse over its domain.
\end{proposition}
\begin{proof}
Since  
$\alpha_s(0)=0$ and $\alpha_s(1)=1$, all we have to do is show that the derivative of $\alpha_s$ with respect to $p$ is strictly positive in $[0,1]$, whence we are done by the inverse function theorem.

Setting $y=\frac{2p}{1+p}$, we can rewrite $\alpha_s(p)$ as $\alpha_s(y)=\frac{2y-y^s}{2-y^s}$. By the chain rule,
\begin{align*}
\frac{d \alpha_s}{d p}(p)&=\left(\frac{d y}{dp}(p)\right) \left(\frac{d\alpha_s}{dy}(y(p))\right)\\
& \frac{2}{(1+p)^2}\frac{\left(4-2sy^{s-1}+2(s-1)y^s\right)}{(2-y^s)^2}. 
\end{align*}
The derivative, with respect to y, of the numerator in the expression above is 
\[2\left(-2s(s-1)(1-y)y^{s-2}\right)\leq 0\qquad (\textrm{since $y\in [0,1]$}).\] 
Thus the minimum of the numerator is attained when $y(p)=1$. In particular, 
\begin{align*}
\frac{d \alpha_s(p)}{d p}(p)&\geq \frac{2}{(1+p)^2}\frac{2}{(2-y^s)^2}>0, 
\end{align*}
as required.
\end{proof}

In general, computing an explicit closed-form expression for the inverse of $\alpha_s$ is difficult, reflecting the fact that computing the probability that the intersection of an $s$-set of $k$-dimensional subcubes chosen uniformly at random 
is non-empty is difficult (or at least unpleasant). It is for this reason that in Theorem~\ref{uniform: below the threshold} we give the thresholds for the uniform model in terms of the thresholds for the binomial model.

\begin{proof}[Proof of Theorem~\ref{uniform: below the threshold}]
The key observation is that we can view the binomial model as the result of a two stage random process. In the first stage we randomly partition the set of vertices $V$ into sets $V_0, V_1 \ldots V_d$, where 
$v\in V$ is included in $V_i$ with probability $\binom{d}{i}p^i(1-p)^{d-i}$ independently at random for each $v,i$. In the second stage for each $k$ we associate independently to each vertex in $V_k$ a feature subcube of dimension $k$ chosen uniformly at random, and then build the subcube intersection graph as normal. In particular, the restriction of $G_{V,d,p}$ to the (random) subset $V_k$ is exactly (an instance of) the uniform model $G_{V_k, d, k}$. We shall use this to pull results back from the binomial model to the uniform model.

 Let $\varepsilon>0$ and $\alpha \in (0,1)$ be fixed, and let $k(d)= \lfloor \alpha d \rfloor$. Let $p=\left(\alpha_s\right)^{-1}(\alpha)$.

\noindent{\textbf{Upper bound}:} Pick $\eta$ with $0<\eta<\alpha$. Let $p'=p'(\alpha-\eta)=\left(\alpha_s\right)^{-1}(\alpha-\eta)$. Consider the binomial random subcube intersection graph $G_{[N],d,p'}$, and let $X= \frac{\log N}{d}$.

Suppose $X= t_{K_s}(p')+\varepsilon+o(1)$. By Corollary~\ref{binomial s constant clique threshold}, $G_{[N],d, p'}$ then whp contains an $s$-clique. Further, by Proposition~\ref{prop: dimension of subcubes in s-cliques}, whp there exists such a clique in which all subcubes have dimension at least $k_-=\lceil(\alpha-2\eta)d\rceil$ and at most $k_+=\lfloor\alpha d\rfloor $.

Let $V'$ denote the set of vertices in $G_{[N], d, p}$ whose feature subcubes have dimension in the range $[k_-, k_+]$.  Since the number of vertices with a feature subcube of a given dimension is a binomial random variable, a standard Chernoff bound shows that, for $d$ large enough, the size of $V'$ is whp at most 
\[N'\leq  3\eta d  N\binom{d}{\lfloor (\alpha-\eta) d\rfloor}{p'}^{\lfloor (\alpha-\eta) d \rfloor} (1-p')^{d-\lfloor (\alpha-\eta) d\rfloor}\]

For each $v\in V'$ with a feature subcube of dimension $k'$ with $k_-\leq k' \leq k_+$, select a $(k_+ -k')$-subset of its fixed coordinates uniformly at random from all possibilities, and change those coordinates to wildcards $\star$ (i.e. to free coordinates). This gives a new feature subcube $f'(v)$ with dimension exactly $k_+$.

We now restrict our attention to the subcube intersection graph $G$ defined by $V'$ and the `lifted' feature subcubes $\left(f'(v)\right)_{v\in V'}$. Observe that the distribution on $k_+$-dimensional subcubes given by $f'$ is exactly the uniform distribution. Thus $G$ is in fact an instance of the uniform model $G_{V',d, k_+}$. Furthermore the `lifting' procedure we performed on the feature subcubes $\left(f(v)\right)_{v\in V'}$ has not destroyed any edge 
--- indeed increasing the dimension of feature subcubes can only add edges --- so that whp $G$ contains an $s$-clique.

It follows that the threshold for the appearance of $s$-cliques in the uniform model with parameter $k_+=\lfloor \alpha d \rfloor$ is at most
\begin{align*}
\frac{\log N'}{d}& \leq \frac{\log N}{d}+\left(\alpha-\eta\right) \log \frac{p'}{\alpha-\eta}+(1-\alpha+\eta)\log \frac{1-p'}{1-\alpha+\eta}+o\left(\frac{\log d}{d}\right)\\
&=t_{K_s}(p')+\left(\alpha-\eta\right) \log \frac{p'}{\alpha-\eta}+(1-\alpha+\eta)\log\frac{1-p'}{1-\alpha+\eta}+\varepsilon +o(1).
\end{align*} 
Since $\varepsilon,\eta>0$ were arbitrary, and since both $p'$ and  $t_{K_s}$ are continuous functions (of $\alpha-\eta$ and $p'=\left(\alpha_s\right)^{-1}(\alpha-\eta)$ respectively), the threshold for the appearance of $s$-cliques in the $k_+$-uniform model is at most
\begin{align*}
\lim_{\varepsilon,\eta\rightarrow 0^+}t_{K_s}(p')&+\left(\alpha-\eta\right) \log \frac{p'}{\alpha-\eta}+(1-\alpha+\eta)\log\frac{1-p'}{1-\alpha+\eta}+\varepsilon \\
&=t_{K_s}(p)+\alpha \log \frac{p}{\alpha}+(1-\alpha)\log \frac{1-p}{1-\alpha},
\end{align*}
proving the claimed upper bound on $T_{K_s}(\alpha)$. (Recall that $p=\left(\alpha_s\right)^{-1}(\alpha)=\lim_{\eta\rightarrow 0}p'$.)

\noindent \textbf{Lower bound:} consider the binomial random subcube intersection graph $G_{[N],d,p}$, and let $X= \frac{\log N}{d}$.

Suppose $X=t_{K_s}(p)-\varepsilon+o(1) $. By Corollary~\ref{binomial s constant clique threshold}, whp $G_{[N],d,p}$ contains no $s$-clique. In particular the subgraph of $G_{[N],d,p}$ induced by the set of vertices $V'$ whose feature subcube have dimension $\lfloor \alpha d \rfloor$ is also $K_s$-free. As we observed, this random subgraph is identical in distribution to the random uniform subcube intersection graph $G_{V', d, k}$. Let $N'=\vert V' \vert$ be the number of vertices it contains.

By a standard Chernoff bound, whp 
\[N' \geq \frac{1}{2}N \binom{d}{\lfloor\alpha d\rfloor}p^{\lfloor \alpha d \rfloor}(1-p)^{d-\lfloor \alpha d \rfloor}.\]
It follows that the threshold for the whp appearance of $s$-cliques in the uniform model $G_{V,d, k}$ is at least 
\begin{align*}
t_{K_s}(p) -\varepsilon+ \alpha \log \frac{p}{\alpha}+(1-\alpha)\log \frac{1-p}{1-\alpha}+o(1).
\end{align*} 
Since $\varepsilon>0$ was arbitrary, the claimed lower bound on $T_{K_s}(\alpha)$ follows.
\end{proof}

\subsection{The covering threshold}\label{subsection: uniform model at covering}
\begin{proof}[Proof of Theorem~\ref{uniform: at the threshold}]
 This is very similar to the proof of Theorem~\ref{binomial: at covering}. Assume without loss of generality that $V=[n]$. We let $\alpha \in (0,1)$ be fixed, set $k=k(d)=\lfloor \alpha d \rfloor$ and consider the uniform random subcube intersection graph $G_{[n], d,k}$.

 Let $\mathbf{0}$ denote the all zero element $(0,0, \ldots, 0)$ from $Q_d$. The expected number of elements of $Q_d=\{0,1\}^d$ not covered by the union of the feature subcubes $\bigcup_{v=1}^n f(v)$ is
 \begin{align*}
 \mathbb{E}\left\vert Q_d\setminus \bigcup_{v=1}^n f(v) \right\vert &= \vert Q_d \vert \mathbb{P}\left(\mathbf{0} \notin \bigcup_{v=1}^n f(v) \right)=2^d \left(1-\frac{1}{2^{d-k}}\right)^n\\
 &=\exp\left(d \log 2 - \frac{n}{2^{d-k}}\left(1+o(1)\right)\right).
 \end{align*}
 Now let $\varepsilon$ be fixed with $0< \varepsilon < \log 2$.
 
 \noindent \textbf{Upper bound:} suppose $n=e^{xd}\geq 2^{d-k}d (\log 2 + \varepsilon)$. Then the expected number of uncovered elements of $Q_d$ is $e^{-\varepsilon d +o(d)}=o(1)$, whence by Markov's inequality we have that whp $\bigcup_{v=1}^n f(v)=Q_d$, as desired.

 \noindent \textbf{Lower bound:} suppose $n=e^{xd}= \lfloor 2^{d-k}d (\log 2 - \varepsilon)\rfloor$. Then the expected number of uncovered elements of $Q_d$ is $e^{\varepsilon d +o(d)}$, which is large, and we use Chebyshev's inequality to show the actual number of uncovered elements is concentrated about this value. As before we compute the expectation of the square of the number of uncovered elements by considering pairs of points lying at Hamming distance $i$ from one another. Let $\mathbf{e}_{[i]}$ denote the element of $Q_d$ with $1$ in the first $i$ coordinates and $0$ otherwise. In the following we take binomial coefficients of negative values to be 0.
 \begin{align*}
 \mathbb{E} \left(\left\vert Q_d\setminus \bigcup_{v=1}^n f(v) \right\vert^2\right)&=2^d\sum_{i=0}^d \binom{d}{i} \mathbb{P}\left( \mathbf{0},\mathbf{e}_{[i]} \notin \bigcup_{v=1}^n f(v) \right)\\
 %
 %
 %
 &= 2^d\sum_{i=0}^d \binom{d}{i}\left(1 - \frac{2}{2^{d-k}}+\frac{\binom{d-i}{k-i}}{\binom{d}{k}}\frac{1}{2^{d-k}}\right)^n\\
 &= e^{2\varepsilon d}\sum_{i=1}^d \frac{\binom{d}{i}}{2^d}\exp\left(\frac{\binom{d-i}{k-i}}{\binom{d}{k}}d(\log 2- \varepsilon)+o(1) \right).
 \end{align*}
 We now bound the sum above just as we did in the proof of Theorem~\ref{binomial: at covering}, to show
 \begin{align*}
  \mathbb{E} \left(\left\vert Q_d\setminus \bigcup_{v=1}^n f(v) \right\vert^2\right)=(1+o(1))e^{2\varepsilon d}.
  \end{align*}
 Since the details are similar, we omit them. We deduce just as in Theorem~\ref{binomial: at covering} that $\textrm{Var}\left\vert Q_d \setminus \bigcup_{v=1}^n f(v)\right\vert=o\left(\mathbb{E} \left\vert Q_d \setminus \bigcup_{v=1}^n f(v)\right\vert^2 \right)$. By Chebyshev's inequality whp at least $e^{\varepsilon d +o(d)}\gg 1$ elements of $Q_d$ are not covered by $\bigcup_{v=1}^n f(v)$, as required. 
 \end{proof}

\subsection{Above the covering threshold}\label{subsection: uniform model above covering}

\begin{proof}[Proof of Theorem~\ref{uniform: above the threshold}]
This is similar to the proof of Theorem~\ref{binomial s>>d/logd}. Without loss of generality, we may assume that $V=[n]$. Fix $\varepsilon>0$ and $\alpha\in (0,1)$. Let $k=k(d)=\lfloor \alpha d \rfloor$, and consider the random subcube intersection graph $G_{[n],d,k}$.

\noindent \textbf{Upper bound:}
this case is in fact easier than for the binomial model. By the Helly property, $G_{[n],d,k}$ contains an $s$-clique if and only if some element $x$ of the ambient hypercube $Q_d$ is contained in at least $s$ feature subcubes. Now in $G_{[n],d,k}$ every feature subcube has dimension $k$, thus certainly if $n\geq \frac{(s-1)\vert Q_d \vert }{2^{k}}+1=(s-1)2^{d-k}+1$, we have by the pigeon-hole principle that some $x\in Q_d$ is contained in at least $s$ feature subcubes, and thus $G_{[n],d,k}$ contains a copy of $K_s$.

\noindent \textbf{Lower bound:} suppose $n \leq (1-\varepsilon)s 2^{d-k}$. Let $\mathbf{0}$ be the all zero element from $Q_d$. The number $C_{\mathbf{0}}$ of feature subcubes containing $\mathbf{0}$ is the sum of $n$ independent identically distributed Bernoulli random variables with parameter $2^{-(d-k)}$.
We have $\mathbb{E} C_{\mathbf{0}}= n 2^{-(d-k)}\leq(1-\varepsilon)s$. Applying a Chernoff bound, we deduce that
\begin{align*}
\mathbb{P}(C_{\mathbf{0}}\geq s)&\leq e^{-\frac{\varepsilon^2}{3}s}.
\end{align*}
In particular the expected number of elements of $Q_d$ contained in at least $s$ feature subcubes is at most $2^de^{-\frac{\varepsilon^2}{3}s}$, which is $o(1)$ for $s\gg d$. It follows by Markov's inequality and the Helly property for subcube intersections that whp $G_{[n],d, k}$ contains no copy of $K_s$.
\end{proof}

\section{Concluding remarks}\label{remarkssection}

\subsection{Small $p$ and large $p$}
In this paper we focussed on the case where $p\in (0,1)$ is fixed (in the binomial model). Let us make here a few remarks about the small $p$ and large $p$ case.

\noindent\textbf{Small $p$:} note first of all that the proofs of Theorems~\ref{binomial s<<d/log d}, \ref{binomial covering threshold} and~\ref{binomial s>>d/logd} all extend to the case when $p=p(d)\rightarrow 0$ as $d\rightarrow \infty$. Similarly, Theorems~\ref{uniform: at the threshold} and~\ref{uniform: above the threshold} for the uniform model also hold when $\alpha=\alpha(d)\rightarrow 0$ as $d\rightarrow \infty$. The proof of Theorem~\ref{uniform: below the threshold} does not, however, go through as it is stated in the paper --- it needs stronger concentration for the dimension of the feature subcubes in Proposition~\ref{prop: dimension of subcubes in s-cliques}.

There are two further remarks worth making concerning the small $p$ case. First of all, as $p$ (or $\alpha$) tends to $0$, the covering results Theorems~\ref{binomial covering threshold} and~\ref{uniform: at the threshold} `converge' to the classical coupon collector problem. Taking $\Omega=Q_d$ and independently drawing random elements from $\Omega$, we expect to make roughly $\vert \Omega \vert \log \vert \Omega \vert= 2^d d \log 2$ draws before we cover $\Omega$, and this is the limit of $\left(\frac{2}{1+p}\right)^d d \log 2$ as $p \rightarrow 0$ (binomial model) and of $2^{(1-\alpha)d}d \log 2$ as $\alpha\rightarrow 0$ (uniform model).

Secondly, for the uniform model with constant parameter $k=1$, the uniform model is closely related to bond percolation on the hypercube, which is a well-studied model of random graph in its own right (see e.g.~\cite{BollobasKohayakawaLuczak92,BorgsChayesVanderHofstadRemcoSladeSpencer06}). On the other hand, the binomial model with parameter $p=\frac{1}{d}$ is different: the dimension of its feature subcubes have an approximatively Poisson distribution, and one does see feature subcubes of large bounded dimension. These will have an impact on the thresholds for lower-order cliques --- indeed, a quick calculation shows that for $s$ 
 fixed, the expected dimension of feature subcubes in $s$-cliques of $G_{V,d,\frac{1}{d}}$ is $2-\frac{1}{d}+O(\frac{1}{d^2})$.

\noindent \textbf{Large $p$:} in this case, we expect quasirandom behaviour from $G_{V,d,p}$ (behaviour `like and Erd{\H o}s--R\'enyi random graph'). We establish it below in the special case when $p=1-\varepsilon(d)$, with $\varepsilon(d)$ of order $\frac{1}{\sqrt{d}}$, when the edge-density is of constant order.
\begin{proposition}\label{quasirandomness}
Let $\varepsilon(d)$ be a sequence of reals from the interval $[0,1]$ with $\varepsilon^2 d$ bounded away from both $0$ and $+\infty$. Then for $p=1-\varepsilon$, with probability tending to $1$ as $n\rightarrow \infty$ the graph $G_{[n],d,p}$ is quasirandom with parameter $e^{-\frac{\varepsilon^2d}{2}}$.
\end{proposition}
\begin{proof}
We shall use the celebrated quasirandomness theorem of Chung, Graham and Wilson~\cite{ChungGrahamWilson89}, which states (amongst other things) that if the number of $K_2$ (\emph{edges}) and the number of $C_4$ ($4-cycles$) contained in a graph are `what you would expect if the graph was a typical Erd\H{o}s--R\'enyi random graph with parameter $q$', then $G$ is quasirandom with parameter $q$ (we refer the reader to~\cite{ChungGrahamWilson89} for a formal definition of quasirandomness).

Let $\varepsilon(d)$ be sequence of positive real numbers as in the statement of the proposition, and let $q=e^{-\frac{\varepsilon^2d}{2}}$. Set $p=1-\varepsilon$, and consider a labelled $4$-set $\{v_1,v_2,v_3,v_4\}$ of vertices from $G_{[n],d,p}$. The probability that $v_1v_2$ forms an edge of the graph is
\[\mathbb{P}(\textrm{edge})=\left(1-\frac{(1-p)^2}{2}\right)^d= e^{-\frac{\varepsilon^2}{2}d+O(\varepsilon^3 d)}=q(1+o(1)),\]
while the probability that all of the edges $v_1v_2$, $v_2v_3$, $v_3v_4$ and $v_4v_1$ are present in the graph is
\[\mathbb{P}(\textrm{$4$-cycle})=\left(2\left(\frac{1+p}{2}\right)^4-p^4+p^2(1-p)^2\right)^d= e^{-2\varepsilon^2d+O(\varepsilon^3 d)}=q^4(1+o(1)).\]

We now verify that the numbers of edges $\#\{K_2\}$ and of $4$-cycles $\#\{C_4\}$ are concentrated about their respective expectations. We appeal to the second moment method once more. For the edge $K_2$ we already established that $\textrm{Var}\#\{K_2\}=o\left(\left(\mathbb{E}\#\{K_2\}\right)^2\right)$ in the proof of Theorem~\ref{binomial s<<d/log d}. Thus by Chebyshev's inequality, we have the required concentration: whp $G_{[n],d,p}$ contains $(1+o(1))\mathbb{E}\#\{K_2\}=(1+o(1))\binom{n}{2}q$ edges. Regarding $4$-cycles, we have
\begin{align*}
\mathbb{E} \left(\#\{C_4\}\right)^2&= \binom{n}{4}3 \left(\binom{n-4}{4}3 \mathbb{P}\left(\textrm{$4$-cycle}\right)^2  +O(n^3)\right)\\
&= \left(\mathbb{E} \#\{C_4\}\right)^2 (1+o(1)),
\end{align*}
by using in the second line the fact that $\mathbb{P}\left(\textrm{$4$-cycle}\right)^2=e^{-2\varepsilon^2 d}\gg \frac{1}{n}$. We deduce that $\textrm{Var}\#\{C_4\}=o\left(\left(\mathbb{E}\#\{C_4\}\right)^2\right)$ and that the number of $4$-cycles in $G_{[n],d,p}$ is concentrated about its mean: whp \[\#\{C_4\}=(1+o(1))\frac{n(n-1)(n-2)(n-3)}{8}q^4.\]
Our proposition then follows from the quasirandomness theorem of Chung, Graham and Wilson~\cite{ChungGrahamWilson89}.
\end{proof}

Thus in this case the binomial model behaves like an Erd\H{o}s--R\'enyi random graph. It is not hard to use this to show that the uniform model with parameter $k=d-O(\sqrt{d})$ is also quasirandom. As such, there is nothing very novel about our models in this range.

\subsection{Further questions}
There are a number of further natural questions to ask about our models.

Two such questions concern connectivity and component evolution.
 We address these in a forthcoming paper~\cite{FalgasRavryMarkstrom13b}, in which we show that the connectivity threshold for the binomial model with $p\leq \frac{1}{3}$ is $t_{\textrm{connect}}=\log \frac{2}{1+p}$, coinciding with the covering threshold. On the other hand, for the range $p> \frac{1}{3}$ we relate the connectivity of the binomial model to that of the uniform model for a suitable choice of parameter $k$.

It remains an open problem to understand independence in the context of subcube intersection graphs. We do not know how to track the independence number of our models, and more generally we do not know how to perform anything but the most basic computations involving non-edges. Similarly, we have a lower bound on the chromatic number coming from the clique number, but no non-trivial upper bound.

Finally, given our motivation for studying subcube intersection graphs, it would be desirable to allow some bias in the distribution of the feature subcubes. For instance, in a polarised society with two-party politics it is likely citizens will have either mostly zeroes (`left-wing opinions') or mostly ones (`right-wing opinions') amongst their opinions. It would then be interesting to study the change in the behaviour of our models as the polarisation becomes stronger.

\section*{Acknowledgements}
The authors would like to thank Joel Larsson for helpful comments on the first version of this article.

\section*{Appendix A: proof of Lemma~\ref{technical lemma}}

\begin{lemma*}
For all $y \in [0,1]$ and all integers $0\leq i \leq s$, the following inequality holds:
\begin{align*}\left(2+2y^i-4y^s+y^{2s-i}\right)^s \leq \left(2-y^s\right)^{2s-i}. \end{align*}
\end{lemma*}
\begin{proof}
This is a simple (albeit lengthy) exercise in calculus. Let $h_1(y)=2+2y^i-4y^s+y^{2s-i}$. Since we may write $h_1(y)$ as
\begin{align*}
h_1(y)=2(1-y^s)+ 2y^i(1-y^{s-i})+y^{2s-i},
\end{align*} 
we have $h_1(y)>0$ for $y\in [0,1]$. The function 
\[g(y)= \frac{\left(2-y^s\right)^{2 s-i}}{\left(2+2 y^i-4 y^s+y^{2s-i}\right)^s}\]
is thus well-defined and differentiable in the interval $[0,1]$.

We want to show that $g(y)\geq 1$ for all $y \in [0,1]$. If $i=s$ or if $i=0$ the function $g$ is identically $1$ on the interval $[0,1]$, in which case we have nothing to prove. Assume therefore in what follows that $0<i<s$.  We have $g(0)=2^{s-i}> 1$ and $g(1)=1$, so we will be done if we can show that the function $g$ is monotone decreasing in the interval $[0,1]$. We compute the derivative of $g$:
\begin{align*}
g'(y)&=\frac{1}{(h_1(y))^{s+1}}\Bigl( -s (h_1(y))' (2-y^s)^{2s-i} + h_1(y) ((2-y^s )^{2s-i})'\Bigr)\\
&= \frac{-sy^{i-1} (2-y^s)^{2s-i-1}}{(h_1(y))^{s+1}}\Bigl(4i -(4s+2i)y^{s-i}+(4s-2i)y^{2s-2i}+ (4s-4i)y^s -(4s-4i)y^{2s-i} \Bigr)\\
&=-\frac{sy^{i-1}(2-y^s)^{2s-i-1} (1-y^{s-i})}{(h_1(y))^{s+1}}\Bigl(4i -(4s-2i)y^{s-i} +(4s-4i)y^s\Bigr). 
\end{align*}
We claim $g'(y)\leq 0$  for all $y \in [0,1]$. Clearly
\[-\frac{sy^{i-1}(2-y^s)^{2s-i-1} (1-y^{s-i})}{(h_1(y))^{s+1}} \leq 0\]
for all $y\in[0,1]$. Thus the only factor we have left to consider is
\begin{align*}
h_2(y)&= 4 i-(4s-2i) y^{s-i}+(4s-4i) y^s.
\end{align*}
\begin{claim*} 
$h_2(y)> 0$ for all $y \in [0,1]$.
\end{claim*}
The claim above implies that $g'(y)\leq 0$ for all $y \in [0,1]$, whence $g(y) \geq g(1)=1$ for all $y\in [0,1]$, as desired. 
\begin{proof}[Proof of Claim] 
We have $h_2(0)=4i$ and $h(1)=2i$, both of which are strictly positive. 
We differentiate $h_2$ to check for other extrema inside the interval $[0,1]$.
\[h_2'(y)=(s-i)y^{s-i-1}\left(-(4s-2i)+4s y^i \right) .\]
In addition to $y=0$, $h_2'$ has one root in the interval $[0,1]$, namely
\[y_{\star}= \left(1-\frac{i}{2s}\right)^{\frac{1}{i}}.\]
At $y_{\star}$, we have
\[h_2(y_{\star})= 4 i\left(1 - \left(1-\frac{i}{2s}\right)^{\frac{s}{i}}\right)>0.\]
Thus for all $y \in [0,1]$, 
\[h_2(y)\geq \min\Bigl(h_2(0), h_2(y_{\star}), h_2(1)\Bigr)>0,\]
establishing the claim. 
\end{proof}
This completes the proof of Lemma~\ref{technical lemma}.
 \end{proof} 
\end{document}